\documentclass[12pt]{amsart}
\usepackage{times}
\usepackage[T1]{fontenc}
\usepackage{dsfont}
\usepackage{mathrsfs}
\usepackage[colorlinks]{hyperref}
\usepackage{xcolor}
\usepackage[a4paper,asymmetric]{geometry}
\usepackage{mathscinet}
\usepackage{fullpage}
\usepackage{latexsym}
\usepackage{amsthm}
\usepackage{amssymb}
\usepackage{amsfonts}
\usepackage{amsmath}
\newtheorem{theorem}{Theorem}[section]
\newtheorem{thm}[theorem]{Theorem}

\newtheorem{lemma}[theorem]{Lemma}
\newtheorem{lem}[theorem]{Lemma}
\newtheorem{proposition}[theorem]{Proposition}
\newtheorem{prop}[theorem]{Proposition}
\newtheorem{corollary}[theorem]{Corollary}

\theoremstyle{definition}
\newtheorem{definition}[theorem]{Definition}
\newtheorem{defn}[theorem]{Definition}

\theoremstyle{remark}
\newtheorem{remark}[theorem]{Remark}
\newtheorem{rem}[theorem]{Remark}
\numberwithin{equation}{section}

 \DeclareMathAlphabet{\mathpzc}{OT1}{pzc}{m}{it}

  \newcommand{\dif}{\mathrm{d}}

 \newcommand{\norm}[1]{\left\Vert#1\right\Vert}
 
 \newcommand{\E}{\mathbb{E}}            

 \newcommand{\N}{\mathbb{N}}
 \newcommand{\R}{\mathbb{R}}\newcommand{\RR}{\mathbb{R}}
 
 \newcommand{\PP}{\mathbb{P}}

 \newcommand{\Be}{\begin{equation}}
 \newcommand{\Ee}{\end{equation}}
 \newcommand{\Bs}{\begin{split}}
 \newcommand{\Es}{\end{split}}
  \newcommand{\Bes}{\begin{equation*}}
 \newcommand{\Ees}{\end{equation*}}
 \newcommand{\BT}{\begin{thm}}
 \newcommand{\ET}{\end{thm}}
 \newcommand{\Bp}{\begin{proof}}
 \newcommand{\Ep}{\end{proof}}
 \newcommand{\BL}{\begin{lem}}
 \newcommand{\EL}{\end{lem}}
 \newcommand{\BP}{\begin{proposition}}
 \newcommand{\EP}{\end{proposition}}
 \newcommand{\BC}{\begin{corollary}}
 \newcommand{\EC}{\end{corollary}}
 \newcommand{\BR}{\begin{rem}}
 \newcommand{\ER}{\end{rem}}
 \newcommand{\BD}{\begin{defn}}
 \newcommand{\ED}{\end{defn}}
 \newcommand{\BI}{\begin{itemize}}
 \newcommand{\EI}{\end{itemize}}

 \newcommand{\al}{\alpha}
 \newcommand{\de}{\delta}
 \newcommand{\be}{\beta}
 
\def\dko{d_{\mathrm{Kol}}}
\def\dw{d_{\mathrm{W}}}
\def\dwb{d_{\mathrm{W}_\beta}}
\def\dfm{d_{\mathrm{FM}}}

\begin{document}
\title[Approximation of stable law by Stein's method]
{Non-integrable stable approximation by Stein's method}

\author[P. Chen]{Peng Chen}
\author[I. Nourdin]{Ivan Nourdin}
\author[L. Xu]{Lihu Xu}
\author[X. Yang]{Xiaochuan Yang}
\author[R. Zhang]{Rui Zhang}


\begin{abstract} \label{abstract}
We develop Stein's method for $\alpha$-stable approximation with $\alpha\in(0,1]$, continuing the recent line of research by Xu \cite{lihu} and Chen, Nourdin and Xu \cite{C-N-X} in the case $\alpha\in(1,2).$ The main results include an intrinsic upper bound for the error of the approximation in a variant of Wasserstein distance that involves the characterizing differential operators for stable distributions, and an application to the generalized central limit theorem. Due to the lack of first moment for the approximating sequence in the latter result, we appeal to an additional truncation procedure and investigate fine regularity properties of the solution to Stein's equation.
\end{abstract}

\maketitle

\section{Introduction}

\subsection{Overview}
Non-Gaussian stable distributions arise ubiquitously in probabilistic approximation of random phenomenon. The most fundamental example is the generalized central limit theorem for the sum of independent random variables with common distribution that presents heavy tails \cite[p.161]{Dur10}.

Recall that a real-valued random variable $Z$ has a stable distribution if it satisfies the distributional identity: for any $a,b\in\RR$, there exist $c,d\in\RR$, such that $aZ_1+ bZ_2 \overset{d}= cZ+d$ where $Z_1,Z_2$ are independent copies of $Z$. It is called strictly stable, if the above relation holds with $d=0$ for any choice of $a,b$.
In this work, we focus on \emph{non-Gaussian strictly stable laws}, whose definition is given equivalently by their characteristic functions as follows.
\begin{definition}\cite[p.86]{sato}
Let $Z$ be a real-valued random variable. It has the non-Gaussian \emph{strictly} stable distribution if there exists $\al\in(0,2)$, $\sigma>0$, $\de\in[-1,1]$ (and $\tau\in\RR$ when $\al=1$) such that for all $\lambda\in\R$,
    \begin{align}\label{e:stable_exponent}
    \mathbb{E}\big[e^{i\lambda Z}\big]= \begin{cases}
   \exp\big\{-\sigma^\alpha|\lambda|^\alpha(1-i\,\delta\,{\rm sign}(\lambda)\,\tan\frac{\pi\alpha}2)\big\}&\quad\mbox{if $\alpha\in(0,2)\setminus\{1\}$},\\
    \exp\big\{-\sigma|\lambda| + i\tau\lambda \big\}&\quad\mbox{if $\alpha=1$.}
    \end{cases}
    \end{align}
Here $\al$ is called the stability parameter determining the decay of the tail of $Z$; $\sigma$ the scale parameter which is the stable analogue of the standard deviation; $\de$ the skewness parameter describing the asymmetry, and $\tau$ the shift parameter describing the mode.
\end{definition}

\begin{rem}
By a change of variables $Z\mapsto Z/\sigma$ in the case $\al\in(0,2)\setminus\{1\}$ or $Z\mapsto (Z-\tau)/\sigma$ in the case $\al=1$, we can and will assume that $\sigma=1$ and $\tau=0$ throughout the paper. As such, the strictly $1$-stable law is symmetric in that $Z=-Z$ in distribution, and the strictly $\al$-stable law with $\al\neq 1$ is asymmetric unless $\de=0$. With this in mind, we adopt the unified notation $Z\sim S_\al(\de)$ for strictly $\al$-stable distributions, with $\de$ being arbitrary in $[-1,1]$ if $\al\neq 1$ and $\de$ tacitly assumed to be $0$ if $\al=1$, corresponding to the symmetric Cauchy distribution.  In this paper, we exclude the case $\de\in\{1,-1\}$ because the crucial Lemma \ref{p-prop} does not hold in such case, see Remark \ref{r:subordinator}.
\end{rem}



This paper is concerned with the proximity between a random variable and a strictly stable distribution in some appropriately chosen distance, and is along the recent line of research initiated by Xu \cite{lihu}, who proved an upper bound in the Wasserstein distance $d_\mathrm{W}$ between an \emph{integrable} random variable $F$ and  $Z\sim S_\al(0)$ with $1<\al<2$. Recall that $d_\mathrm{W}(F,Z):= \sup_{h\in\mathcal H} |\E[h(F)]-\E[h(Z)]|$ where the supremum runs over all Lipchitz functions with Lipchitz constant at most $1$. Xu showed
\begin{align}\label{e:Xu}
d_{\mathrm{W}}(F,Z) \le \sup_{f\in\mathcal F}|\E[\Delta^{\frac{\alpha}{2}}f(F)]-\frac{1}{\al}\E[Ff'(F)]|,
\end{align}
where $\Delta^{\frac \al 2}$ is the fractional Laplacian normalized in such a way that its Fourier multiplier is $|\lambda|^\al$, and $\mathcal F$ is the class of functions with first and second derivatives bounded by a generic constant depending on $\al$. Later on, Chen, Nourdin and Xu \cite{C-N-X}
showed that a similar estimate holds for asymmetric stable distributions with $1<\al<2$ where the fractional Laplacian was replaced by other relevant non-local operators.  This bound turned out to be effective in giving rates of convergence for the generalized central limit theorem, see \cite[Th. 2.6]{lihu} and \cite[Th. 1.4]{C-N-X}.

The derivation of \eqref{e:Xu} relies on tools and ideas developed by Ch. Stein in the seventieth for \emph{normal approximation}. Stein's idea is so robust that can be naturally extended to other target distributions, Poisson, Gamma, Beta, to name just a few. We refer the reader to the webpage \cite{yvik-web} maintained by Y. Swan for an exhaustive list.  Roughly speaking,  Stein's  method for approximating a probability distribution $\mu$ is composed of three main steps. First, one characterizes $\mu$ by certain ``differential" operator $\mathcal{A},$ namely $\mathbb{E}[\mathcal{A}f(Y)]=0$ for an appropriate class of functions $f$ if and only if $Y$ is distributed according to $\mu$. Second, one establishes and solves a ``differential" equation,
\begin{equation}\label{e:Stein}
\mathcal{A}f(x)=h(x)-\mathbb{E}[h(Y)],
\end{equation}
 called Stein's equation, so that $\mathbb{E}[h(F)]-\mathbb{E}[h(Y)]=\mathbb{E}[\mathcal{A}f(F)]$ for any random variable $F$.   Third, one studies the regularity property of the solution $f$ to Stein's equation with respect to the regularity property of the function $h$, which will be used to estimate the expectation $\E[\mathcal A f(F)]$, thus providing the distance between $F$ and $Y$.
In view of these steps, the bound \eqref{e:Xu} hinges on the fact \cite[Th. 4.1]{lihu} that $\Delta^{\frac \al 2} f(x)  -\frac{1}{\al}xf'(x)$ characterizes a symmetric stable law with $1<\al<2$. Also, the class $\mathcal F$ in the bound captures regularity properties of solutions to Stein's equation \eqref{e:Stein}.

\medskip
The goal of this paper is to carry out Stein's method for $\al$-stable approximation with $0<\al\le 1$, continuing and completing the study of \cite{lihu,C-N-X}. For $Z\sim S_\al(\de)$, one prominent difference between the cases $\al\le 1$ and $\al>1$ is that $\E[|Z|]<\infty$ if and only if $\al>1$. Therefore, in the case $\al\le 1$, the random variable of interest $F$ in the approximation, e.g. sum of independent random variables with common distribution that has non-integrable tails, typically does not possess finite first moment. The usual Wasserstein distance has to be changed in order to obtain meaningful bounds.  In this paper, we make use  of
\begin{align}\label{e:dwb}
d_{\mathrm{W}_\beta}(F,Z):= \sup_{h\in\mathcal H_\beta} |\E[h(F)] - \E[h(Z)]|, \quad \be<\al,
\end{align}
where $\mathcal H_\beta$ is the class of Lipschitz functions $h:\R \to (\R,d_\beta)$ such that $|h(x)-h(y)|\le d_\beta(x,y)$, the real line being endowed with the metric $d_\beta(x,y)=|x-y|\wedge |x-y|^\beta$. It is known \cite[p.18]{ST} that $\E[|Z|^\beta]<\infty$ whenever $\beta<\al$, thus $\dwb$ is an appropriate metric for determining the proximity between  $F$ and \emph{non-integrable} stable laws. The advantage of using $\dwb$ over the (non-smooth) Kolmogorov distance $\dko$ given by
\begin{align*}
\dko(F,Z):=\sup_{x\in\RR}|\PP(F\le x)-\PP(Z\le x)|
\end{align*}
is that the solution $f$ to Stein's equation for $h\in \mathcal H_\be$ possess nice regularity properties in terms of boundedness and H\"older continuity of the (fractional)  derivatives of $f$, which are crucial for obtaining rates of convergence in the generalized CLT. It's worthy to point out that $\dwb$ and $\dko$ satisfy a similar relation to that of $\dw$ and $\dko$, see Corollary \ref{c:SCLT} below.

Another complication  for approximating $Z\sim S_\al(\de)$ with $\al\le 1$ arises in applying a bound like \eqref{e:Xu} to obtain the generalized CLT. The approach of \cite{lihu} or \cite{C-N-X} consists in applying regularity properties of the solution to Stein's equation to obtain a stochastic Taylor-like expansion which, combined with Stein's $K$-function or leave-one-out argument,  yields the desired rates of convergence for the generalized CLT.  Such a program amounts to significant changes in the case $\al\le 1$, e.g. finer (compared to the case $\al>1$) regularity behavior of the solution to Stein's equation has to be established which may have merits on its own, and an additional truncation term has to be handled.

  The discussion around rates of convergence in the generalized CLT has been made by many authors. Using the distance $\dko$, Hall \cite{Hall82} gave two-sided bounds (tight under some assumption), see also \cite{JuPa98} where the same distance was used. In \cite{KuKe00,DaNa02} were discussed the $L^\infty$ or $L^1$ rate of convergence of the density function of the approximating sequence. All these efforts are made upon analysis of the characteristic function of the partial sum. By extending Lindeberg's approach and applying Dynkin's formula, Chen and Xu \cite{C-X} gave a real-variable proof for the generalized CLT with explicit rates. Recently, Arras and Houdr\'e \cite{AH} developped Stein's method for self-decomposable laws with finite first moment. The requirement of finite first moment excludes $\al$-stable distributions with $\al\in(0,1]$. To the best of our knowledge, the present paper is the first dealing with Stein's method for nonintegrable random variables. In view of far-reaching applications of Stein's method  \cite{BaCeXi07,BrDa17,ChaSha11,GoTi06,KuTu11,Fang14,GaPiRe17,Hsu05,NoPe09,bluebook,NoPeSw14,F-S-X},  we believe that the present work opens the way to non-integrable stable approximation in a context where the dependence structure of the random phenomenon of interest is complex, e.g. random graphs, interacting particle systems etc.

\medskip
As already said, throughout the paper, we assume $Z\sim S_\al(\de)$ for $0<\al\le 1$ and $\de\in(-1,1)$, with the convention that $\de=0$ if $\al=1$. We use $c,C$ to denote generic constants, and $c_\al,C_\al$ constants that depend only the subscript, whose value may change in each appearance. All random objects below are defined on some common probability space.

\subsection{Statement of results} \label{s:MThm}

We start with some notation.  
Set for $\de\in(-1,1)$ and $0<\al\le 1$
\begin{align*}
\nu_{\al,\de}(dx)= {d_\al}\left( \frac{1+\de}{2} x^{-1-\al}{\bf 1}_{(0,\infty)}(x) + \frac{1-\de}{2}|x|^{-1-\al}{\bf 1}_{(-\infty,0)}(x) \right)dx ,
\end{align*}
where
\begin{align*}
0<d_\al = \begin{cases}
\frac{-1}{\Gamma(-\al)\cos \frac{\pi\al}{2}} &\mbox{ if } 0<\al<1, \\
 \frac 2{\pi} & \mbox{ if } \al=1.
\end{cases}
\end{align*}
Then, by applying  \cite[Eq. (14.18) and (14.20)]{sato} twice, one has
\begin{align*}
&\int_\RR (e^{i\lambda x} -1) \nu_{\al,\de}(dx) = -|\lambda|^\al \left(1-i\de \rm{sign}(\lambda)\tan \frac{\pi\al}{2}\right), \quad 0<\al<1;\\
&\int_\RR (e^{i\lambda x}-1-i\lambda x{\bf 1}_{|x|<1}) \nu_{1,0}(dx)= -|\lambda|.
\end{align*}
Therefore, $\nu_{\al,\de}$ and $\nu_{1,0}$ are the L\'evy measures associated with $S_\al(\de)$ and $S_1(0)$, respectively.  As such, strictly stable distributions are infinitely divisible with zero Gaussian coefficient and zero drift, see \cite[p.37]{sato} for the general form of an infinitely divisible distribution.
Define the operator
\begin{align}\label{op}
\mathcal L^{\al,\de} f(x) = \begin{cases} \int_\RR  (f(x+u) -f(x)) \nu_{\al,\de}(du) & \mbox{ if } 0<\al<1 \\
\int_\RR [f(x+u) - f(x) - uf'(x) 1_{|u|\le 1}] \nu_{1,0}(du) & \mbox{ if } \al=1
\end{cases}
\end{align}
for any sufficiently smooth function $f$. Recall that $\mathcal L^{\al,\de}$ is the generator of a stable L\'evy process $(Z_t)_{t\ge 0}$ with $Z_1\sim S_\al(\de)$, see \cite[p.208]{sato}. When $\de=0$, $\mathcal L^{\al,0}$ is the usual fractional Laplacian.

Our first result is the following, providing an intrinsic upper bound  in the spirit of \eqref{e:Xu} for the $\dwb$-distance between $S_\al(\de)$ and any random variable $F$.

\begin{thm}\label{thm1}
i) Let $Z\sim S_\al(\de)$ with $0<\al< 1$. For $\be<\al$, we have
\begin{align}\label{e:steinstable}
\dwb(F,Z)\le \sup_{f\in \mathcal F_\be}|\E[\mathcal L^{\al,\de}f(F)]- \frac{1}{\al}\E[Ff'(F)]|
\end{align}
where each element of the class $\mathcal F_\be$ satisfies the following.
\begin{itemize}
\item[a.] $\big|f(x)-f(y)\big|\leq c_{\alpha,\beta}|x-y|\wedge|x-y|^{\beta}.$
\item[b.] $\norm{f'}_{\al}\le c_\al$, where $\norm{g}_\al:=\norm{g}_\infty + \sup_{x\neq y}\frac{|g(x)-g(y)|}{|x-y|^\al}$.
\item[c.] $\norm{\mathcal L^{\al,\de}f}_\gamma\le c_{\al,\be,\gamma}$ for any $\gamma\in(0,1)$.
\end{itemize}
ii) Let $Z\sim S_1(0)$ and $\be<1$. The bound \eqref{e:steinstable} holds with $\mathcal L^{\al,\de}$ replaced by $\mathcal L^{1,0}$, and $\mathcal F_\be$ replaced by $\mathcal F_1$ composed of functions satisfying
\begin{itemize}
 \item[a.] $\big|f(x)-f(y)\big|\leq c_{\alpha,\beta}|x-y|\wedge|x-y|^{\beta}.$
 \item[b.] $\norm{f'}_{1,\log}\le c$, where $\norm{g}_{1,\log}= \norm{g}_\infty +\sup_{x\neq y, |x-y|\le 1} \frac{|g(x)-g(y)|}{|x-y|(1-\log|x-y|)}$.
 \item[c.] $\norm{\mathcal L^{1,0} f}_{1,\log}\le c$.
\end{itemize}
\end{thm}

\begin{rem} The theorem is proved by solving Stein's equation and studying the regularity of the solutions. This is achieved  by Barbour's generator approach \cite{B} where one is concerned with a Markov process with stationary distribution being the stable law. We give details in Section \ref{s:OU}. Recalling the class $\mathcal F$ in \eqref{e:Xu}, the properties satisfied by elements in $\mathcal F_\be$ involve less differentiability, but more on the H\"older continuity. As first sight, the obtained upper bound may be infinite for non-integrable $F$. It turns out that such bound gives nearly optimal rates in many examples by carefully making use the regularity estimates.
\end{rem}

To illustrate an explicit use of our abstract Theorem \ref{thm1}, we compute the rate of convergence in the generalized central limit theorem for the partial sum of a sequence of independent and identically distributed random variables in the domain of normal attraction of an $\alpha$-stable law  defined as follows.

\begin{defn}\label{px}
A real-valued random variable $X$ is said to be in the domain of normal attraction $\mathcal{D}_\alpha$  of an $\al$-stable law if its cumulative distribution function $F_X$ has the form
\begin{equation}\label{dis}
1-F_X(x) = \frac{A+\epsilon(x)}{|x|^{\alpha}}(1+\delta) \quad \mbox{ and } \quad  F_X(-x)=\frac{A+\epsilon(-x)}{|x|^{\alpha}}(1-\delta)
\end{equation}
as $x>1$, where $\alpha\in(0,1]$, $\delta\in[-1,1]$, $A>0$,  and $\epsilon: \R\to\R$ is a bounded function vanishing at $\pm\infty$. We write for simplicity $X\in \mathcal D_\al$. Since $\epsilon$ is a bounded function, we denote
\begin{align*}
K:=\sup_{x\in\mathbb{R}}|\epsilon(x)|<\infty.
\end{align*}
\end{defn}


\begin{theorem}\label{thm2}
Let $X_1,X_2,\ldots$ be independent and identically distributed random variables defined on a common probability space. Suppose that $X_1$ has
a distribution of the form (\ref{dis}) and there exists a positive constant $L,$ such that for any $|x|>L,$ the $\epsilon(x)$ in (\ref{dis}) is $C^2$, which satisfies $x\epsilon'(x)\rightarrow0$ as $|x|\rightarrow \infty$.
Set
$\sigma=\left(2A\alpha/d_\alpha \right)^\frac1\alpha$ and
\begin{align}\label{sum}
S_{n}=\frac{1}{\sigma}n^{-\frac{1}{\alpha}}
\begin{cases}
X_{1}+\cdots+X_{n}, &\alpha\in(0,1),\\
X_{1}+\cdots+X_{n}-n\mathbb{E}\big[X_{1}{\bf 1}_{(0,\sigma n)}(|X_{1}|)\big], &\alpha=1.
\end{cases}
\end{align}
\noindent
i) When $\alpha\in(0,1),$ we have
\begin{align*}
&\dwb\big(S_n,Z\big)\\
\leq&
c_{\alpha,\beta,A,K}
\Big[n^{-1}\!+\!n^{-\frac{1}{\alpha}}\!\int_{-\sigma n^{\frac{1}{\alpha}}}^{\sigma n^{\frac{1}{\alpha}}}\!|x|^{1+\alpha}\big|d\frac{\epsilon(x)}{|x|^{\alpha}}\big|\!+\!
n^{\frac{\alpha-\beta}{\alpha}}\!\int_{|x|\geq \sigma n^{\frac{1}{\alpha}}}|x|^{\beta}\big(\big|d\frac{x\epsilon'(x)}{|x|^{\alpha}}\big|\!+\!\big|d\frac{\epsilon(x)}{|x|^{\alpha}}\big|\big)\!+ \!\mathcal{R}_{1,n}\Big] ,
\end{align*}
where $c_{\alpha,\delta,\beta,A,K}$ is a constant which depends on $\alpha,\delta,\beta,A,K$ that can be made explicit and
\begin{align*}
\mathcal{R}_{1,n}=&\sup_{|x|\geq\sigma n^{\frac{1}{\alpha}}}|\epsilon(x)|+n^{\frac{\alpha-1}{\alpha}}\Big|(1+\delta)\int_{0}^{\sigma n^{\frac{1}{\alpha}}}\frac{\epsilon(x)}{x^{\alpha}}\dif x-(1-\delta)\int_{0}^{\sigma n^{\frac{1}{\alpha}}}\frac{\epsilon(-x)}{x^{\alpha}}\dif x\Big|.
\end{align*}

\noindent
ii) When $\alpha=1$ and $\delta=0,$ we have
\begin{align*}
&\dwb(S_n,Z)\\
\leq&c_{\beta,A,K}\Big[n^{-1}(\log n)^2
+n^{-1}\int_{-\sigma n}^{\sigma n}|x|^{2}\big(2-\log|\frac{n^{-1}}{\sigma}x|\big)\big|d\frac{\epsilon(x)}{|x|}\big|\\
&\qquad\qquad\qquad\qquad+n^{1-\beta}\int_{|x|\geq \sigma
n}|x|^{\beta}\big(\big|d\frac{x\epsilon'(x)}{|x|}\big|+\big|d\frac{\epsilon(x)}{|x|}\big|\big)+\mathcal{R}_{2,n}\Big],
\end{align*}
where
\begin{align*}
\mathcal{R}_{2,n}=n^{-1}(\log n)^{2}\Big|\int_{0}^{\sigma n}\frac{\epsilon(x)-\epsilon(-x)}{x}\dif x\Big|.
\end{align*}
\end{theorem}

\begin{rem}
Let us explain heuristically why these bounds can be naturally derived from \eqref{e:steinstable}.  Recall the tail probability of $Z\sim S_\al(\de)$ (see \cite[p.16]{ST})
\begin{align*}
&\PP(Z> x) \sim  \frac{d_\al}{\al} \frac{1+\de}{2} x^{-\al} = \nu_{\al,\de}([x,\infty))\\
&\PP(Z< -x) \sim \frac{d_\al}{\al} \frac{1-\de}{2}  x^{-\al}= \nu_{\al,\de}((-\infty,-x])
\end{align*}
as $x\to\infty$. We formally regard $\mathcal L^{\al,\de}f$ in \eqref{e:steinstable} as a weighted increment of $f$ (the density of $Z$ being the weight function) which, together with regularity properties of $f$ and Taylor-like expansion, is comparable with $S_nf'(S_n)$. The fact that the common distribution of the sequence does not behave exactly as $Z$ in their tails makes the role of $\epsilon$ in these bounds clear. That said, more precise expansion for the tail behavior of $Z$ may be used to improve the bounds.  It is plain that the decay at infinity of $\epsilon$ affects the decay of the above bounds.
\end{rem}

Moreover, by the same argument as the proof of \cite[Corollary I.1]{C-D}, we get the following convergence upper bound on $\dko(S_n,Z),$ which will be proved in the appendix.
\begin{corollary}  \label{c:SCLT}
Keep the same notation and assumptions as in Theorem \ref{thm2}. Then
\begin{align}\label{kol}
\dko(S_n,Z)\leq C_{\alpha}\Big[\dwb\big(S_n,Z\big)\Big]^{\frac{1}{2}}.
\end{align}
\end{corollary}

In Theorem \ref{thm2}, we observe that the function $\epsilon$ is required to satisfy $\epsilon(x)\rightarrow0$ as $x\rightarrow\pm\infty.$ But that $\epsilon$ vanishes is not a necessary condition for the generalized CLT to hold. Actually, by slightly modifying the approach leading to Theorem \ref{thm2}, we can also consider examples where $\epsilon$ is a slowly varying function diverging at infinity. Because it would be too technical to state such result at a great level of generality, we prefer to illustrate an explicit situation for which our methodology still allows to conclude in appendix, and we will give a proof that rather relies on the density function.

\vskip 3mm

\subsection{Examples and application}

We present several consequences of our main results and compare them with those previously obtained for stable approximation.

In our first example, we consider an  independent sequence with common Pareto distributions, namely,
\begin{align*}
\mathbb{P}(X_{1}> x)=\frac{1+\delta}{2|x|^{\alpha}},\quad x\geq1,\qquad \mathbb{P}(X_{1}\leq x)=\frac{1-\delta}{2|x|^{\alpha}},\quad x\leq-1,
\end{align*}
for $\alpha\in(0,1]$ and $\delta\in[-1,1],$ whose sum scaled by $n^{-\frac{1}{\alpha}}$ weakly converges to a stable distribution. In the case $\delta=0,$ that is, $\alpha$-stable distribution is symmetric, the authors of \cite{DaNa02} proved a rate $n^{-\frac{\alpha}{1+\alpha}}$  in total variation distance and conjectured that a better rate  $n^{-1}$ in
total variation distance should be valid. Our result gives a partial answer to their conjecture, that the rate $n^{-1}$ is valid for the $\dwb$-distance. In the case $\alpha=1,$ our result is within a $\log n$ factor of optimal.

\smallskip
The second example concerns a sequence of independent random variables with common distribution function
\begin{align*}
&\mathbb{P}(X_{1}> x)=(A|x|^{-\alpha}+\widetilde{A}|x|^{-\widetilde{\alpha}})(1+\delta),\quad x\geq1,\\
&\mathbb{P}(X_{1}\leq x)=(A|x|^{-\alpha}+\widetilde{A}|x|^{-\widetilde{\alpha}})(1-\delta),\quad x\leq-1,
\end{align*}
for some $A>0,\tilde{A}>0$ and $\tilde{\alpha}>\alpha.$
We obtain the convergence rate $n^{-1}+n^{\frac{\al-\tilde{\al}}{\al}}$ for $\al\in(0,1)$ which is the same convergence rate as \cite{KuKe00} in Kolmogorov distance for the case $\delta=0$.

The third example is suggested by Persi Diaconis, who proposed it as an open problem in AMS workshop 'Stein's method and its application in high dimensional statistics' in August 2018 \cite{Dia}. This problem originates from the fact $Y=^{d}\frac{1}{Y}$ if $Y$ follows the symmetric Cauchy distribution $S_1(0)$.
So we consider the reciprocal of sum of random variables and obtain the convergence rate in Kolmogorov distance.

In Appendix \ref{s:rv}, we also consider an example
where the common distribution has the regularly varying density $p(x)=\frac{\alpha^2e^\alpha}{2(1+\alpha)}\,\frac{\log |x|}{|x|^{\alpha+1}}{\bf 1}_{[e,\infty)}(|x|)$, which is not in the domain of normal attraction of a stable law. We obtain the same convergence rate $(\log n)^{-1}$ as  \cite{JuPa98} for the case
$p(x)=\frac{\alpha^2e^\alpha}{(1+\alpha)}\,\frac{\log |x|}{|x|^{\alpha+1}}{\bf 1}_{[e,\infty)}(x)$.

\subsection{Plan of the paper}

The rest of the paper is organized as follows.  Section \ref{s:prelim} is devoted to some preliminary facts and Barbour's generator approach for solving Stein's equation.  In Section \ref{s:reg_Stein}, we study the solution to Stein's equation in detail, and obtain the intrinsic bound Theorem \ref{thm1} as a byproduct. Making use of the regularity results obtained in Section \ref{s:reg_Stein}, we prove Theorem \ref{thm2} in Section \ref{s:proofs}. Examples are worked out in Section  \ref{s:examples} and auxiliary results are given in Appendix.

{\bf Acknowledgements:} We would like to gratefully thank Jay Bartroff, Larry Goldstein, Stanislav
Minsker, and Gesine Reinert for organizing a stimulating workshop \cite{Dia}, and the participants for several discussions. We would also like to gratefully thank Persi Diaconis for very helpful discussions and encouraging us to add his example in the paper.

\section{Preliminaries}\label{s:prelim}

\subsection{Properties of stable densities}

We recall some basic facts on the densities of a strictly $\al$-stable L\'evy process $(Z_t)_{t\ge 0}$ with $Z_1\sim S_\al(\de)$. Denote by $p_t(x)$ the density of $Z_t$. It is known \cite{C-Z} that $p(t,x,y):=p_t(y-x)$ is the fundamental solution of the operator $\mathcal L^{\al,\de}$ in the sense that
\begin{align}\label{e:p_PDE}
\partial_t p(t,x,y) + \mathcal L^{\al,\de} p(t,\cdot,y)(x)=0.
\end{align}
The self-similarity of the process $(Z_t)_{t\ge 0}$ translates to the following scaling relation
\begin{align}\label{e:p_scaling}
p_t(x) = c^{\frac{1}{\al}} p_{ct}(c^{\frac 1 \al}x),  \quad \forall c>0, t>0, x\in\RR.
\end{align}
Write for simplicity $p_1(x)=p(x)$. In the following lemma, we list a few estimates of the stable densities that will be useful for our purposes.

\begin{lemma}\label{p-prop} Suppose that $\de\not\in\{-1,1\}$. The following statements hold for all $t>0$, $x\in\RR$.
\begin{itemize}
  \item [(1)]$$p_t(x)\le \frac{C_{\alpha}t}{(t^{1/\alpha}+|x|)^{\alpha+1}}.$$
  \item [(2)]$$|\mathcal{L}^{\alpha,\delta}p_t(x)|\le\frac{C_{\alpha}}{(t^{1/\alpha}+|x|)^{\alpha+1}}.$$
  \item[(3)] $$|p_t(x)-p_t(y)|\le C_{\alpha}\left(\frac{|x-y|}{t^{1/\al}}\wedge 1\right)(p_t(x)+p_t(y)).$$
  \item [(4)] In the case $\de=0$, we have for $k\in\N$, $$|\partial^{k}_x p(t,x)|\le\frac{C_{\alpha}t}{(t^{1/\alpha}+|x|)^{\alpha+1+k}}.$$
\end{itemize}
\end{lemma}

\begin{proof}
By the scaling property, we only need to consider $t=1$.
When $\delta=0,$ the $\alpha-$stable process is symmetric, so by \cite[Theorem 1.1 and Lemma 2.2]{ChZh16}, we immediately obtain the results for $\alpha=1.$ Now we consider $\alpha\in(0,1)$.
(1) and (3) are from
\cite[(2.7) and (3.1)]{C-Z},
respectively.
By the same arguments as the proof of \cite[(2.28)]{ChZh16}
\begin{align*}
\mathcal{L}^{\alpha,\delta}p(t,x)&=d_{\alpha}\int_{-\infty}^{\infty}\frac{p(t,x+y)-p(t,x)}{2|y|^{1+\alpha}}\big((1+\delta){\bf 1}_{(0,\infty)}(y)+(1-\delta){\bf 1}_{(-\infty,0)}(y)\big)\dif y\\
&\leq d_{\alpha}\int_{-\infty}^{\infty}\frac{|p(t,x+y)-p(t,x)|}{|y|^{1+\alpha}}\dif y\leq\frac{C_{\alpha}}{(t^{1/\alpha}+|x|)^{\alpha+1}},
\end{align*}
which implies (2).
\end{proof}

\begin{remark} \label{r:subordinator}
The condition on $\de$ is necessary when $\al<1$, as one can see from the specific case $S_{1/2}(1)$, where $p_t(x)=\frac{t}{x^{3/2}}e^{-\frac{t^2}{x}}1_{x>0}$ is the density of a $\frac 1 2$-stable subordinator.
\end{remark}

\subsection{Distance $\dwb$}

By Kantorovich duality \cite[p.19]{V-C}, the distance $\dwb$ used in the paper is the Wasserstein distance corresponding to the cost function $d_\be(x,y)=|x-y|\wedge|x-y|^\be$ with $\be<\al$ such that the distance is meaningful for $\al$-stable approximation.
Though not explicitly said, it is easy to see that functions in the space $\mathcal H_\be$ are \emph{bounded} due to the following lemma.
\begin{lem}\label{l:2.1}
Let $h\in\mathcal H_\be$ with $\be<\al$. Then $|h(y)-\E[h(Z)]|\le C_{\alpha}(1+|y|^{\beta})$.
\end{lem}
\begin{proof} Let $x\mapsto p(x)$ be the density of $Z\sim S_\al(\de)$.
Write 
\begin{align*}
|h(y)-\E[h(Z)]|&\leq\int p(x)|h(y)-h(x)|dx\leq \int p(x)\big(|x|^{\beta}+|y|^{\beta}\big)dx\leq C_{\alpha}(1+|y|^{\beta}),
\end{align*}
where the last inequality thanks to Lemma \ref{p-prop}.
\end{proof}
Recall the Fortet-Mourier distance (see \cite[Section C.2]{bluebook} )
\begin{align*}
\dfm(X,Y):=\sup_{h\in \mathcal H_\mathrm{FM}}\big|\E[h(X)]-\E[h(Y)]\big|,
\end{align*}
where the class $\mathcal H_\mathrm{FM}$ contains  all Lipschitz functions $h$ such that $\sup_{x\in\R}|h(x)|+\sup_{x\in\R}|h'(x)|\le 1$.
It is plain that $d_{\mathrm{FM}}\le 2\dwb$. In view of the fact that the Fortet-Mourier distance metrizes convergence in distribution of random variables, $\dwb$ is  suitable for assessing the rate of convergence of limit theorems.

\subsection{Solving Stein's equation by Barbour's generator approach}\label{s:OU}
First we show that stable distributions are characterized by the operator 
\begin{align}\label{e:operator_OU}
\mathcal{A}_{\al,\de}f(x)=\mathcal{L}^{\alpha,\delta}f(x)-\frac{1}{\alpha}xf'(x).
\end{align}
Note that $\mathcal{A}_{\al,\de}$ is the generator of the Markov process solving the following Orenstein-Uhlenbeck type stochastic differential equation
\begin{align}\label{e:OU}
\begin{cases} X_{t}=\int_0^t -\frac{1}{\alpha}X_{s}ds+ Z_{t} \\
X_0=x
\end{cases},
\end{align}
where $(Z_{t})_{t\geq0}$ is an $\alpha$-stable L\'evy process with $Z_1\sim S_\al(\de)$, see \cite[Th. 6.7.4]{A}.  Such an equation can be solved explicitly
\begin{align}\label{e:OU_explicit}
X^x_t = xe^{-\frac t \al} + \int_0^t e^{-\frac {t-s} \al} dZ_s,
\end{align}
see \cite[p.105]{sato}, and provides an interpolation between any Dirac mass and $S_\al(\de)$.  The corresponding semigroup $(Q_{t})_{t\geq0}$ can then be used to solve Stein's equation.



\begin{prop}
Let $Y$ be a real-valued random variable. Then $Y\sim S_\al(\de)$ if and only if $\E[\mathcal A_{\al,\de} f(Y)]=0$ for all $f\in C^\infty$ and fast decaying at infinity.
\end{prop}
\begin{proof}
Let $Z\sim S_\al(\de)$ and $(X^x_t)_{t\ge 0}$ be the unique pathwise solution \eqref{e:OU_explicit} to the SDE \eqref{e:OU}. Since the L\'evy measure of $Z$ satisfies the condition (17.11) of \cite[Th. 17.1]{sato}, $X^x_t$ converges in distribution as $t\to\infty$ to a random variable with characteristic function $\lambda\mapsto \exp[\int_0^\infty \psi(e^{-\frac{s}{\al}}\lambda) ds]$, where $\psi$ is the principal log of the characteristic function of $Z$. It is readily checked that this characteristic function coincides with \eqref{e:stable_exponent}. As a consequence, $S_\al(\de)$ is the unique invariant distribution of the semigroup $(Q_t)_{t\ge 0}$ associated with $\mathcal A_{\al,\de}$ by \cite[Cor. 17.9]{sato}.  Denote by $\mu$ the distribution of $Z$. We have
\begin{align*}
\int f(x)\mu(dx) = \iint f(y) Q_t(x,dy) \mu(dx)
\end{align*}
for any $t\ge 0$ and smooth $f$. Taking derivative with respect to $t$ at $t=0$ yields $\E[\mathcal A_{\al,\de} f(Z)]=0$, as desired.

Now assume that $\E[\mathcal A_{\al,\de} f(Y)]=0$ for all smooth $f$, namely, the distribution of $Y$ is the infinitesimal invariant distribution of $(Q_t)_{t\ge 0}$ in the sense of Albeverio, Ruediger and Wu \cite{ARW}.  In the symmetric case $\de=0$, such a condition implies that $Y\sim S_\al(0)$, see \cite[Prop. 3.2]{ARW}. This statement  continues to hold in the asymmetric case $\de\neq 0$. We prove this in Appendix \ref{a:albevio}.
\end{proof}

Recall that the Ornstein-Uhlenbeck type process $(X^x_t)_{t\ge 0}$ can also be represented by a time-changed stable L\'evy process. To see this, set $Y_t:= \int_0^t e^{\frac {s} \al} dZ_s$ and $V_t= Y_{\log(1+t)}$. Then $(V_t)_{t\ge 0}$ has independent increments because $(Z_t)_{t\ge 0}$ does. On the other hand, one can prove that
\begin{align}\label{e:V_si}
\E[e^{i\lambda(V_t-V_s)}] = (\E[e^{i\lambda Z}])^{t-s}
\end{align}
for any $t\ge s\ge 0$. One concludes that $(V_t)_{t\ge 0}\overset{d}=(Z_t)_{t\ge 0}$.  In view of \eqref{e:OU_explicit}, one has
\begin{align*}
X^x_t \overset{d}= xe^{-\frac t \al} + e^{-\frac t \al} Z_{e^t-1} \overset{d}= xe^{-\frac t \al} +  Z_{1-e^{-t}},
\end{align*}
where we used the self-similarity of $(Z_t)_{t\ge 0}$ in the second identity. It follows that the transition density of $(Q_t)_{t\ge 0}$, namely the density of $X^x_t$, is given by
\begin{equation}\label{density}
  q(t,x,y)=p_{1-e^{-t}}(y-e^{-t/\alpha}x)=s(t)^{-1/\alpha}p(s(t)^{-1/\alpha}(y-e^{-t/\alpha}x)),
\end{equation}
where $s(t)=1-e^{-t}$ and we used again the self-similarity of $(Z_t)_{t\ge 0}$. The proof of \eqref{e:V_si} is given in Appendix \ref{a:V_si}.

\medskip


%

Now we consider and solve Stein's equation $$\mathcal{A}_{\al,\de}f(x)=\mathcal{L}^{\alpha,\delta}f(x)-\frac{1}{\alpha}xf'(x)=h(x)-\mathbb{E}\big[h(Z)\big]$$ for $h\in \mathcal H_\be.$
Lemma \ref{lm23} below may be explained by semigroup interpolation argument. The operator $\mathcal{A_{\al,\de}}$ generates the semigroup $(Q_{t})_{t\geq0}$ whose transition density satisfies \eqref{density}. It follows that $Q_{0}h=h$ and $Q_{\infty}(h)=\mathbb{E}\big[h(Z)\big].$ Thus, setting $f=-\int_{0}^{\infty}\big(Q_{t}h-Q_{\infty}h\big)dt$ for an appropriate class of $h,$ we have $\mathcal{A_{\al,\de}}f=-\int_{0}^{\infty}\mathcal{A}_{\al,\de}Q_{t}hdt=-\int_{0}^{\infty}\partial_{t}Q_{t}hdt=Q_{0}h-Q_{\infty}h,$ as desired.  In Appendix \ref{a:solve_Stein}, we give a detailed proof.



\begin{lemma}\label{lm23}
Let $Z\sim S_\al(\de)$ with $0<\al\le 1$ and $h\in\mathcal H_\be$ with $0<\be<\al$.
Set
\begin{align}
f(x)&:=-\int_{0}^{\infty}\mathbb{E}\big[h\big(X_{t}^{x}\big)-\E h(Z)\big]dt,\nonumber \\
& = -\int_0^\infty \int
p_{1-e^{-t}}(y-e^{-\frac{t}{\alpha}}x)(h(y)-\E h(Z))dydt \nonumber \\
& = -\int_0^\infty \int p(y)\left[ h((1-e^{-t})^{1/\al}y+e^{-t/\al} x)- h(y)\right] dy dt.\label{phih}
\end{align}
Then
\begin{equation}\label{stablestein}
\mathcal{A}_{\al,\de}f(x)= h(x) - \E h(Z).
\end{equation}
\end{lemma}

Note that the last two identities follow from \eqref{density} and a change of variables.
We end this section by verifying that \eqref{phih} is well-defined.
Since $h\in\mathcal H_\be$, we have
\begin{align*}
&\big|h((1-e^{-t})^{1/\al}y+e^{-t/\al} x)- h(y)\big| \\
&\le e^{-t/\al}|x| + e^{-t\be/\al} |x|^\be +  |y(1-(1-e^{-t})^{1/\al})|\wedge |y(1-(1-e^{-t})^{1/\al})|^\be,
\end{align*}
which is integrable with respect to $1_{t>0}dt\otimes p(y)dy$, as desired.
\section{Study of Stein's equation and Proof of Theorem \ref{thm1}}\label{s:reg_Stein}

\subsection{The regularity estimates of the solution $f$}
Theorem \ref{thm1} follows immediately once some regularity properties of the solution to Stein's equation are in place. Let us state these results first and prove them in Section \ref{sr}. Some of the results below actually play a crucial role in the proof of Theorem \ref{thm2}. Indeed, we are going to make use of the regularity estimates in order to control the error induced by the leave-one-out argument.

\vskip 3mm
$\bullet$ \underline{{\bf $\alpha\in(0,1)$}} {\bf :}
\begin{proposition}\label{prop3}
Let $\alpha \in (0,1)$. For any $h \in \mathcal H_\be$ with $\beta
\in (0,\alpha)$, let $f$ be defined as (\ref{phih}). Then the
following statements hold:

\noindent (i). We have
\begin{equation}  \label{e:F'LeAlp}
\|f'\|_{\infty} \ \le \ \alpha,
\end{equation}
\begin{equation}  \label{e:HolF'}
\sup_{x \neq y} \frac{|f'(x)-f'(y)|}{|x-y|^{\alpha}}\ \le \ C_{\alpha}.
\end{equation}
(ii) For any $x,w\in\mathbb{R},$
\Be  \label{e:Holf}
  |f(x+w)-f(x)|\ \le C_{\alpha,\beta}|w|\wedge|w|^{\beta},
\Ee
\Be  \label{e:FraFBou}
  \|\mathcal{L}^{\alpha,\delta}f\|_{\infty} \ \le C_{\alpha,\beta}.
\Ee
For any $\gamma \in (0,1)$,
\begin{align}\label{regularity1}
|\mathcal{L}^{\alpha,\delta}f(x)-\mathcal{L}^{\alpha,\delta}f(y)|\le C_{\alpha,\beta,\gamma}|x-y|^{\gamma}.
\end{align}
\end{proposition}

\vskip 3mm
$\bullet$ \underline{{\bf $\alpha=1$}} {\bf :}
\begin{proposition}\label{P:fc}
Let $\alpha=1$. For any $h \in \mathcal H_\be$ with $\beta \in (0,1)$ and $f$ is defined as (\ref{phih}), we have\\
(i)
\begin{equation}  \label{e:F'LeAlpc}
\|f'\|_{\infty} \ \le \ 1,
\end{equation}
and for any $|x-z|<1$,
\begin{equation}\label{e:holdf'c}
  |f'(x)-f'(z)|\le C\Big(2-\log |x-z|\Big)|x-z|.
\end{equation}
(ii) For any $x,w\in\mathbb{R},$
\Be  \label{e:Holf1}
  |f(x+w)-f(w)|\ \le C_{\beta}|w|\wedge|w|^{\beta},
\Ee
\Be\label{e:FraFBouc}
\|\mathcal{L}^{1,0}f\|_\infty\le C_{\beta}.
\Ee
For any $|x-y|<1$,
\Be \label{e:FraFHolc}
|\mathcal{L}^{1,0}f(x)-\mathcal{L}^{1,0}f(y)|\le C_{\beta}|x-y|\left(1-\log|x-y|\right).
\Ee
\end{proposition}

\subsection{Proof of Theorem \ref{thm1}}
Let $Z\sim S_\al(\de)$, one has
\begin{align*}
d_{\mathrm{W}_\beta}(X,Z) &= \sup_{h\in \mathcal H_\beta} |\E h(X) -\E h(Z) | \le \sup \E [\mathcal A_{\al,\de} f(X)]
\end{align*}
where the supremum runs over all the functions $f$ of the form \eqref{phih} with $h\in\mathcal H_\be$ and Parts a.-c.  follow from Propositions \ref{prop3} and \ref{P:fc}. \qed

\subsection{Proof of Propositions \ref{prop3} and \ref{P:fc}}\label{sr}

Now we prove Propositions \ref{prop3} and \ref{P:fc} through some lemmas below.

\begin{lemma}\label{lem:f'}
Let $\alpha \in (0,1]$. For any $h \in \mathcal H_\be$ with $\beta
\in (0,\alpha)$, let $f$ be defined as (\ref{phih}).

 (1) We have
$$\|f'\|_\infty\le \alpha.$$

(2) If $\alpha\in(0,1)$, then
\begin{equation*}
  |f'(x)-f'(z)|\le C_{\alpha}|x-z|^{\alpha}.
\end{equation*}
If $\alpha=1$, then for any $|x-z|<1$,
\begin{equation*}
  |f'(x)-f'(z)|\le C\Big(2-\log |x-z|\Big)|x-z|.
\end{equation*}
\end{lemma}
\noindent\textbf{Proof:}
(1) Note that $\norm{h'}_\infty\le 1$, from which it is readily checked that one can differentiate under the integral sign in \eqref{phih}. Hence
\begin{align}\label{e:f'}
f'(x)= -\int_0^\infty\int e^{-t/\al} p(y)h'(s(t)^{1/\al} y + e^{-t/\al} x)dy dt,
\end{align}
yielding $|f'(x)|\le \al$ for all $x\in\R$.

(2)
Choose $B=|x-z|^\alpha$. Applying successively \eqref{e:f'}, change of variables, and Lemma \ref{p-prop}, we get that
\begin{align*}
  &\quad|f'(x)-f'(z)|\\
  &\le\int_0^\infty e^{-t/\alpha} \int_{\R}|p(y-s(t)^{-1/\alpha}e^{-t/\alpha}x)-p(y-s(t)^{-1/\alpha}e^{-t/\alpha}z)||h'(ys(t)^{1/\alpha})|\,dy\,dt\\
  &\le C_{\alpha}\|h'\|_\infty\int_0^\infty e^{-t/\alpha} ((s(t)^{-1/\alpha}e^{-t/\alpha}|x-z|)\wedge1)\,dt\\
  &\le C_{\alpha}\Big(\int_0^B e^{-t/\alpha} dt+\int_B^\infty e^{-2t/\alpha} s(t)^{-1/\alpha}dt|x-z|\Big)\\
  &\le C_{\alpha}\Big(B+\int_B^\infty t^{-1/\alpha}e^{-t/\alpha} dt|x-z|\Big),
\end{align*}
where in the forth inequality, we use the fact that $s(t)^{-1/\alpha}e^{-t/\alpha}=(e^t-1)^{-1/\alpha}\le t^{-1/\alpha}$.
If $\alpha\in(0,1)$, then
$$|f'(x)-f'(z)|\le C_{\alpha}\Big(B+\int_B^\infty t^{-1/\alpha} dt|x-z|\Big)\leq C_{\alpha}|x-z|^{\alpha}.$$
If $\alpha=1$, then for $B=|x-z|<1$,
\begin{align*}
 |f'(x)-f'(z)|&\le C\Big(B+\Big(\int_B^1 t^{-1}dt+\int_1^\infty e^{-t}\,dt\Big)|x-z|\Big)\\
 &\le C\Big(2-\log |x-z|\Big)|x-z|.
\end{align*}
\hfill$\Box$

\begin{lemma}\label{lem:ph}
Let $\alpha \in (0,1]$ and $ h\in \mathcal H_\be$ with $\beta \in (0,\alpha)$. Let $f$ be
defined as \eqref{phih}. If $\alpha\in(0,1)$, then for any $x,w\in\mathbb{R},$
\begin{align*}
|f(x+w)-f(x)|\leq C_{\alpha,\beta}|w|\wedge|w|^{\beta},
\end{align*}
\begin{align*}
\|\mathcal{L}^{\alpha,\delta}f\|_\infty\le C_{\alpha,\beta}.
\end{align*}
If $\alpha=1$ and $\delta=0,$ then for any $x,w\in\mathbb{R},$
\begin{align*}
|f(x+w)-f(x)|\leq C_{\beta}|w|\wedge|w|^{\beta},
\end{align*}
\begin{align*}
\|\mathcal{L}^{1,0}f\|_\infty\le C_{\beta}.
\end{align*}
\end{lemma}
\noindent\textbf{Proof:}
For $\alpha \in (0,1]$, one has by \eqref{phih}
\begin{multline*}
 f(x+w)-f(x) \\ =-\int_0^\infty \int_{\R}p(z)(h(s(t)^{-1/\alpha}z+e^{-t/\alpha}(x+w))-h(s(t)^{-1/\alpha}z+e^{-t/\alpha}x))\,dz\,dt.
\end{multline*}
Thus, for $h \in\mathcal H_\be$ with $\beta \in (0,\alpha)$,
\begin{equation*}
  |f(x+w)-f(x)|\le \int_0^\infty \int_{\R}p(z)e^{-\beta t/\alpha}\,dz\,dt(|w|^\beta\wedge|w|)=\frac{\alpha}{\beta}(|w|^\beta\wedge|w|).
\end{equation*}
It follows that , for $\alpha \in (0,1)$,
$$|\mathcal{L}^{\alpha,\delta}f(x)|\le d_{\alpha}\int_{\R}\frac{|f(x+w)-f(x)|}{|w|^{1+\alpha}}\,dw\le \frac{\alpha d_\alpha}{\beta} \int_{\R}\frac{|w|^\beta\wedge|w|}{|w|^{1+\alpha}}\,dw\le C_{\alpha,\beta}.$$

Now it remains to bound $\norm{\mathcal L^{1,0}f}_\infty$.  By Lemma \ref{lem:f'}, for $|w|\le 1$, one has
\begin{align*}
|f(x+w)-f(x)-f'(x)w| &\le \int_0^{|w|} |f'(x+u)-f'(x)| du \\
&\le C\int_0^{|w|} u(2+\log(1/u)) du\le C w^2 \log (1/|w|),
\end{align*}
It follows  that
\begin{align*}
  |\mathcal{L}^{1,0}f(x)|&\le C\Big(\int_{|w|\le1} \log (1/|w|)\,dw+\int_{|w|>1}|w|^{\beta-2}\,dw\Big)\leq C_{\beta}.
\end{align*}
\qed


\begin{lemma}\label{lem3}
Let $\alpha \in (0,1]$ and $h \in \mathcal H_\be$ with $\beta \in (0,\alpha)$. \\
(1) If $\alpha\in(0,1)$, then  for any $a>0$,
$$\left|\int_{\R}\mathcal{L}^{\alpha,\delta}p(y)h(ay)\,dy\right|\le C_{\alpha,\beta}a^{\alpha}.$$
(2) If $\alpha=1$ and $\delta=0$, then for any $a>0$,
$$\left|\int_{\R}\mathcal{L}^{1,0}p(y)h(ay)\,dy\right|\le C_{\alpha,\beta}(a^{\beta}+a).$$
\end{lemma}
\noindent\textbf{Proof:}
(1)
Let $k_\delta(x)=(1+\delta)\textbf{1}_{(0,\infty)}(x)+(1-\delta)\textbf{1}_{(-\infty,0]}(x)$.
By Fubini's theorem (justified by the fact that $\norm{h}_\infty<\infty$, see Lemma \ref{l:2.1}), we have that,
for any $a>0$,
\begin{align*}
  \left|\int_{\R}(\mathcal{L}^{\alpha,\delta}p)(z)h\left(az\right)\,dz\right|&\leq d_{\alpha}\left|\int_{\R}\int_{\R}
  \frac{(p(z+w)-p(z))k_\delta(w)}{2|w|^{1+\alpha}}
  h\left(az\right)\,dw\,dz\right|\\
  &=d_{\alpha}\left|\int_{\R}\int_{\R}\frac{p(z)k_\delta(w)}{2|w|^{1+\alpha}}(h(az-aw)-h(az))\,dw\,dz\right|\\
  &\le d_{\alpha}a^{\alpha}\int_{\R}\,dw\int_{\R}\frac{p(z)|w|\wedge|w|^\beta}{|w|^{1+\alpha}}\,dz\leq C_{\alpha,\beta}a^{\alpha}.
\end{align*}
(2) Note that
\begin{align*}
  &\mathcal{L}^{1,0}p(x)=\frac{d_{1}}{2}\int_{|w|>1}\frac{p(x+w)-p(x)}{w^{2}}\,dw+\frac{d_{1}}{2}\int_{|w|\le 1}\frac{p(x+w)-p(x)-p'(x)w}{w^2}\,dw.
\end{align*}
By Fubini's theorem,
\begin{align*}
  &\quad\left|\int_{\R}\int_{|w|>1}\frac{p(z+w)-p(z)}{|w|^{2}}h\left(az\right)\,dw\,dz\right|\\
 &=\left|\int_{|w|>1}\int_{\R}\frac{p(z)}{|w|^{2}}(h(az-aw)-h(az))\,dz\,dw\right|\\
 &\le a^{\beta}\int_{|w|>1}\,dw\int_{\R}\frac{p(z)|w|^\beta}{|w|^{2}}\,dz\leq C_{\beta}a^{\beta}.
\end{align*}
Applying Fubini's theorem, integration by parts and the estimate of $p'(x)$ (Lemma \ref{p-prop}), we get
\begin{align*}
  &\quad\left|\int_{\R}\int_{|w|\le 1}\frac{p(z+w)-p(z)-p'(z)w}{|w|^{2}}h\left(az\right)\,dw\,dz\right|\\
 &=\left|\int_{|w|\le 1}\frac{1}{|w|^{2}}\,dw\int_{\R}\left(\int_0^w(p'(z+u)-p'(z))\,du\right)h\left(az\right)\,dz\right|\\
 &=a\left|\int_{|w|\le 1}\frac{1}{|w|^{2}}\,dw\int_{\R}\left(\int_0^w(p(z+u)-p(z))\,du\right)h'\left(az\right)\,dz\right|\\
 &\le Ca\int_{|w|\le 1}\frac{1}{|w|^{2}}\,dw\int_{\R}\left(\int_0^{|w|}|u|(p(z+u)+p(z))\,du\right)\,dz\leq Ca.
\end{align*}
Thus, the assertion is proved.
\hfill$\Box$

\begin{lemma}\label{lem:tria-p}
Let $\alpha \in (0,1)$ or $\alpha=1$ with $\delta=0.$ Then
\begin{equation*}
\int_{\R}\left|\mathcal{L}^{\alpha,\delta}p(z)\right|\,dz\le C_{\alpha}
\end{equation*}
\end{lemma}
\noindent\textbf{Proof:}
If $\alpha\in(0,1)$, then by Lemma \ref{p-prop} (3),
\begin{align*}
&\int_{\R}\left|\mathcal{L}^{\alpha,\delta}p(z)\right|\,dz\le d_{\alpha}\int_{\R}\int_{\R}\frac{|p(z+w)-p(z)|}{|w|^{1+\alpha}}\,dw\,dz\\
   &\le C_{\alpha}\int_{\R}\frac{|w|\wedge1}{|w|^{1+\alpha}}\,dw\int_{\R}p(z+w)+p(z)\,dz\leq C_{\alpha}.
\end{align*}
If $\alpha=1$ and $\delta=0$, then by Lemma \ref{p-prop} (4), we get that, for any $|u|\le 1$,
$$|p''(z+u)|\le \frac{C}{(1+|z+u|)^{4}}\le \frac{C}{(1+|z|)^{4}},$$
where in the last inequality, we use the fact that $2(1+|z+u|)\ge  2+|z|-|u|\ge1+|z|$.
It follows that, for any $|w|\le 1$,
$$|p(z+w)-p(z)-p'(z)w|\le \frac{C}{(1+|z|)^{4}}w^2.$$
Thus, we have that
\begin{align*}
 \int_{\R}\left|\mathcal{L}^{1,0}p(z)\right|\,dz\le& \frac{d_{1}}{2}\int_{\R}\,dz\int_{|w|>1}\frac{p(z+w)+p(z)}{w^{2}}\,dw\\
 &+\frac{d_{1}}{2}\int_{\R}\,dz\int_{|w|\le 1}\frac{|p(z+w)-p(z)-p'(z)w|}{w^2}\,dw\\
 \le& 2d_{1}+\frac{d_{1}}{2}\int_{\R}\,dz\int_{|w|\le 1}\frac{C}{(1+|z|)^{4}}\,dw\leq C.
\end{align*}
\hfill$\Box$

\begin{lemma}\label{lip:T}
Let $\alpha\in(0,1]$ and $h\in\mathcal H_\be$ with $\beta\in(0,\alpha)$.
\begin{itemize}
  \item [(1)]If $\alpha\in(0,1)$, then for any $\gamma\in(0,1)$,
\begin{align*}
|\mathcal{L}^{\alpha,\delta}f(x)-\mathcal{L}^{\alpha,\delta}f(y)|\le C_{\alpha,\beta,\gamma}|x-y|^{\gamma}.
\end{align*}
  \item [(2)]If $\alpha=1$ and $\delta=0$, then for any $|x-y|\le 1$,
\begin{align*}
|\mathcal{L}^{1,0}f(x)-\mathcal{L}^{1,0}f(y)|\le C_{\beta}|x-y|\left(1-\log|x-y|\right).
\end{align*}
\end{itemize}
\end{lemma}
\noindent\textbf{Proof:} (1) Set $s(t)=1-e^{-t}$ and  $\widetilde h = h - \E[h(Z)]$. We claim that
\begin{eqnarray}\label{1.2}
  &&\mathcal{L}^{\alpha,\delta}f(x)=-\int_0^\infty \int_{\R} \mathcal{L}^{\alpha,\delta}q(t,\cdot,y)(x)\widetilde h(y)\,dy\,dt\nonumber\\
  &=&-\int_0^\infty s(t)^{-1}e^{-t}\,dt\int_{\R}(\mathcal{L}^{\alpha,\delta}p)(z)\widetilde h\left(s(t)^{1/\alpha}z+e^{-t/\alpha}x\right)\,dz.
\end{eqnarray}
The second equality follows from \eqref{density}. To see that the first one holds,  note that $h\in \mathcal H_\be$ so that Fubini's theorem implies
\begin{align*}
\mathcal L^{\al,\de} f(x) = -\int_0^\infty \mathcal L^{\al,\de} \left(\int q(t,\cdot,y) \widetilde h(y) dy\right) (x) dt.
\end{align*}
For each fixed $t>0$, applying Lemma \ref{p-prop} (3) justifies a further use of Fubini's theorem, we are led to
\begin{align*}
\mathcal L^{\al,\de} \left(\int q(t,\cdot,y) \widetilde h(y) dy\right) (x) = \int \mathcal L^{\al,\de} q(t,\cdot,y)(x)\widetilde h(y) dy
\end{align*}
and the claim follows.
 By Lemma \ref{lem:tria-p}, we get that
\begin{align*}
  &\quad\left|\int_{\R}(\mathcal{L}^{\alpha,\delta}p)(z)(\widetilde h(s(t)^{1/\alpha}z+e^{-t/\alpha}x)-\widetilde h(s(t)^{1/\alpha}z+e^{-t/\alpha}y))\,dz\right|\\
&\le e^{-t/\alpha}|x-y|\int_{\R}\left|(\mathcal{L}^{\alpha,\delta}p)(z)\right|\,dz
   \le C_{\alpha}|x-y|.
\end{align*}
By Lemma \ref{lem3}  applied to $\widetilde h(\cdot+e^{-t/\al}x), \widetilde h(\cdot+e^{-t/\al}y)\in \mathcal H_\be$, we get that,
\begin{align*}
\left|\int_{\R}(\mathcal{L}^{\alpha,\delta}p)(z)\widetilde h\left((s(t)^{1/\alpha}z+e^{-t/\alpha}x\right)\!dz\!-\!\!\int_{\R}(\mathcal{L}^{\alpha,\delta}p)(z)\widetilde h\left((s(t)^{1/\alpha}z+e^{-t/\alpha}y\right)\!dz\right|\le C_{\alpha,\beta}s(t).
\end{align*}
Thus, we get that, for any $\gamma\in[0,1]$,
\begin{align}\label{1.1}
  &\left|\int_{\R}(\mathcal{L}^{\alpha,\delta}p)(z)\widetilde h\left(s(t)^{1/\alpha}z+e^{-t/\alpha}x\right)\,dz-\int_{\R}(\mathcal{L}^{\alpha,\delta}p)(z)\widetilde h\left(s(t)^{1/\alpha}z+e^{-t/\alpha}y\right)\,dz\right|\nonumber\\
  &\le C_{\alpha,\beta}\big(s(t)\wedge|x-y|\big)\le C_{\alpha,\beta}s(t)(1\wedge s(t)^{-1}|x-y|)^{\gamma}\leq C_{\alpha,\beta,\gamma}s(t)^{1-\gamma}|x-y|^{\gamma}.
\end{align}
Then, by \eqref{1.2} and \eqref{1.1}, we get that
\begin{align*}
 \quad|\mathcal{L}^{\alpha,\delta}f(x)-\mathcal{L}^{\alpha,\delta}f(y)|&\le C_{\alpha,\beta,\gamma}\int_0^\infty s(t)^{-1}s(t)^{1-\gamma}e^{-t}\,dt |x-y|^{\gamma}\\
  &\le C_{\alpha,\beta,\gamma}\int_0^\infty s(t)^{-\gamma}e^{-t}\,dt |x-y|^{\gamma}\\
  &\le C_{\alpha,\beta,\gamma}\int_0^\infty t^{-\gamma}e^{-(1-\gamma)t}\,dt |x-y|^{\gamma}\leq C_{\alpha,\beta,\gamma}|x-y|^{\gamma}.
\end{align*}
(2) Case $\alpha=1$ and $\delta=0$. In the following, we assume that $|x-y|\le 1$.\\
By Lemma \ref{lem:tria-p}, we get that
\begin{align*}
  &\left|\int_{\R}(\mathcal{L}^{1,0}p)(z)(h(s(t)^{1/\alpha}z+e^{-t/\alpha}x)-h(s(t)^{1/\alpha}z+e^{-t/\alpha}y))\,dz\right|\\
&\le e^{-t}|x-y|\int_{\R}\left|(\mathcal{L}^{1,0}p)(z)\right|\,dz.
\end{align*}
By Lemma \ref{lem3}, we get that, for $t<1$,
\begin{align}\label{eq:p2}
 &\quad\left|\int_{\R}(\mathcal{L}^{1,0}p)(z)h\left(s(t)^{1/\alpha}z+e^{-t/\alpha}x\right)\,dz-\int_{\R}(\mathcal{L}^{1,0}p)(z)h\left(s(t)^{1/\alpha}z+e^{-t/\alpha}y\right)\,dz\right|\nonumber\\
&\le C_{\beta}(s(t)+s(t)^{\beta})\le C_{\beta}s(t)^{\beta}.
\end{align}
Let $B=|x-y|^{1/\beta}$. Then by \eqref{1.2}, we get that
\begin{align*}
|\mathcal{L}^{1,0}f(x)-\mathcal{L}^{1,0}f(y)|&\le C_{\beta}\Big(\int_0^B s(t)^{-1}s(t)^{\beta}e^{-t}\,dt+\int_B^\infty s(t)^{-1}e^{-2t}\,dt |x-y|\Big)\\
&\le C_{\beta}\Big(B^{\beta}+\int_B^\infty t^{-1}e^{-t}\,dt |x-y|\Big)\\
&=C_{\beta}\Big(B^{\beta}+\left(\int_B^1 t^{-1}\,dt+\int_1^\infty e^{-t}\,dt\right) |x-y|\Big)\\
&\leq C_{\beta}|x-y|\left(1-\log|x-y|\right),
\end{align*}
ending the proof.
\hfill $\Box$

\section{Proof of Theorem \ref{thm2}}\label{s:proofs}

\subsection{Alternate expressions for $\mathcal L^{\al,\de}$}

The following proposition gathers useful alternate expressions for the operator $\mathcal{L}^{\alpha,\delta}$. 

\begin{prop}\label{properties}
Let $\alpha\in(0,1]$ and
 $f\in C^1(\R)$.
We have, for all $x\in\R$ and $a>0$,

\noindent
a.) When $\alpha\in(0,1),$
\begin{eqnarray*}
(\mathcal{L}^{\alpha,\delta}f)(x)&=&\frac{d_\alpha}{\alpha}\int_0^\infty
\frac{(1+\delta)f'(x+u)-(1-\delta)f'(x-u)}{2u^\alpha}du\notag\\
&=& \frac{a^{1-\alpha}}{\alpha}\int_\R uf'(x+au) \nu_{\al,\de}(du)
\end{eqnarray*}
provided that $\int_{\R}\frac{|f(x+t)-f(x)|}{|t|^{1+\alpha}}\dif t<\infty$ and $\int_{\R}\frac{|f'(x+t)|}{|t|^\alpha}\dif t<\infty$.

\noindent
b.) When $\alpha=1,$ $\delta=0$,
\begin{align*}
\quad(\mathcal{L}^{1,0}f)(x)&=d_{1}\int_{0}^\infty\frac{(f'(x+t)-f'(x){\bf 1}_{(0,1)}(t))-(f'(x-t)-f'(x){\bf 1}_{(0,1)}(t))}{2t}dt\\
&=\int_{-\infty}^\infty t(f'(x+t)-f'(x){\bf 1}_{(-1,1)}(t)) \nu_{1,0}(dt)
\end{align*}
provided that $\int_{\R}\frac{|f(x+t)-f(x)-f'(x){\bf 1}_{(-1,1)}(t))|}{|t|^2}\dif t<\infty$ and $\int_{\R}\frac{|f'(x+t)-f'(x){\bf 1}_{(-1,1)}(t))|}{|t|}\dif t<\infty$.
\end{prop}
\noindent
{\it Proof}.
Note that the conditions on $f$ ensure that all the integrals are well defined and we can use Fubini's theorem in the following proof.
We first consider $\alpha\in(0,1).$\\
1. One can write
\begin{eqnarray}\label{1op}
&&\frac1{d_\alpha}(\mathcal{L}^{\alpha,\delta}f)(x)\nonumber\\
&=&(1+\delta)\int_0^\infty \frac{du}{2u^{1+\alpha}}\int_0^udt\,f'(x+t)
-(1-\delta)\int_{-\infty}^0  \frac{du}{(-u)^{1+\alpha}}\int_u^0 dt\,f'(x+t)\nonumber\\
&=&(1+\delta)
\int_0^\infty dt\,f'(x+t)\int_t^\infty \frac{du}{2u^{1+\alpha}}
-(1-\delta)
\int_{-\infty}^0 dt\,f'(x+t)\int_{-\infty}^t \frac{du}{2(-u)^{1+\alpha}}\nonumber\\
&=&\frac{1}{\alpha}\int_0^\infty
\frac{(1+\delta)f'(x+t)-(1-\delta)f'(x-t)}{2t^\alpha}dt.
\end{eqnarray}
2. One can write
\begin{eqnarray}\label{2op}
\frac1{d_\alpha}(\mathcal{L}^{\alpha,\delta})(x)
&=&\frac{1+\delta}{\alpha}
\int_0^\infty f'(x+t)\frac{dt}{2t^\alpha}
-\frac{1-\delta}{\alpha}
\int_{-\infty}^0 f'(x+t)\frac{dt}{2|t|^\alpha}\nonumber\\
&=&\frac{1}{\alpha}\int_\R tf'(x+t)\frac{(1+\delta){\bf 1}_{(0,\infty)}(t)+(1-\delta){\bf 1}_{(-\infty,0)}(t)}
{2|t|^{\alpha+1}}dt\nonumber\\
&=&\frac{a^{1-\alpha}}{\alpha}\int_\R uf'(x+au)\frac{(1+\delta){\bf 1}_{(0,\infty)}(u)+(1-\delta){\bf 1}_{(-\infty,0)}(u)}
{2|u|^{\alpha+1}}du.
\end{eqnarray}

Now we deal with $\alpha=1$ and $\delta=0$.
We have that
\begin{align}
\quad\frac{1}{d_1}(\mathcal{L}^{1,0}f)(x)&=\int_0^\infty \frac{du}{2u^{2}}\int_0^u\,(f'(x+t)-f'(x){\bf 1}_{(0,1)}(u))\dif t\nonumber\\
&-\int_{-\infty}^0  \frac{du}{2(-u)^{2}}\int_u^0 \,(f'(x+t)-f'(x){\bf 1}_{(-1,0)}(u))\dif t)\nonumber\\
&=\int_0^\infty \frac{f'(x+t)-f'(x){\bf 1}_{(0,1)}(t)}{2t}\dif t
-\int_{-\infty}^0 \frac{f'(x+t)-f'(x){\bf 1}_{(-1,0)}(t)}{2|t|}\dif t\label{3op}\\
&=\int_{-\infty}^\infty
\frac{t(f'(x+t)-f'(x){\bf 1}_{(-1,1)}(t))}{2|t|^2}dt\label{4op},
\end{align}
combining (\ref{3op}) and (\ref{4op}), we immediately obtain the results in the case $\alpha=1.$
\qed
\medskip

\subsection{Taylor-like expansion}
In order to prove Theorem \ref{thm2}, we shall make use of the following lemmas. Recall the definition of $k_\de$ in the proof of Lemma \ref{lem3}.

$\bullet$ \underline{{\bf $\alpha\in(0,1)$}} {\bf :} For any $\delta\in[-1,1],$ we have
\begin{align*}
\int_{\mathbb{R}}\frac{y{\bf 1}_{(-1,1)}(y)}{2|y|^{1+\alpha}} k_\de(y)
 dy=\frac{\delta}{1-\alpha},
\end{align*}
which follows that
for any $a>0,$ we have
\begin{align}\label{ashift}
&\frac{1}{d_{\alpha}}\mathcal{L}^{\alpha,\delta}f(x)-\frac{\delta f'(x)}{\alpha(1-\alpha)}=\frac{1}{\alpha}\int_\R \big(uf'(x+u)-u{\bf 1}_{(-1,1)}(u)f'(x)\big)\frac{k_\de(u)}
{2|u|^{1+\alpha}}du\nonumber\\
=&a^{1-\alpha}\frac{1}{\alpha}\int_\R \big(uf'(x+au)-u{\bf 1}_{(-1,1)}(au)f'(x)\big)\frac{k_\de(u)}
{2|u|^{1+\alpha}}du.
\end{align}
According to (\ref{ashift}), we have the following Taylor-like expansion.
\begin{lem}\label{ml2}
Consider $\alpha\in(0,1).$ Let $X$ have a distribution $F_{X}$ with the form (\ref{dis}),
and $\tilde{X}$ have a distribution $F_{\tilde{X}}$ defined in \eqref{Fun}. $Y$ is a random variable, which is independent with $X$ and $\tilde{X}$.
For any $0<a\leq(2A)^{-\frac{1}{\alpha}}\wedge L\wedge1$ and $f$ defined as above, denote
\begin{align*}
T_{1}\!:=\!\Big|\mathbb{E}\big[Xf'(Y+aX)\big]\!-\!\mathbb{E}[X{\bf 1}_{(-1,1)}(aX)]\mathbb{E}[f'(Y)]-\frac{2A\alpha^{2}}{d_{\alpha}}a^{\alpha-1}\mathbb{E}[\mathcal{L}^{\alpha,\delta}f(Y)-\frac{\delta d_{\alpha}f'(Y)}{\alpha(1-\alpha)}]\Big|,
\end{align*}
then
\begin{align*}
T_{1}\leq C_{\alpha,\beta,A}\Big[a^{\alpha}+a^{\alpha}\int_{-a^{-1}}^{a^{-1}}|x|^{1+\alpha}\big|d\frac{\epsilon(x)}{|x|^{\alpha}}\big|
+a^{\beta-1}
\int_{|x|\geq a^{-1}}|x|^{\beta}\big(\big|d\frac{x\epsilon'(x)}{|x|^{\alpha}}\big|+\big|d\frac{\epsilon(x)}{|x|^{\alpha}}\big|\big)\Big].
\end{align*}
\end{lem}
\begin{proof}
We have by (\ref{ashift})
\begin{align*}
&\frac{2A\alpha^{2}}{d_{\alpha}}a^{\alpha-1}
\mathbb{E}\left[\mathcal{L}^{\alpha,\delta}f(Y)-\frac{\delta d_{\alpha}f'(Y)}{\alpha(1-\alpha)}\right]\\
=&2A\alpha\mathbb{E}\Big[\int_{-\infty}^{\infty}\big[uf'(Y+au)-u{\bf 1}_{(-1,1)}(au)f'(Y)\big]\frac{k_\de(u)}
{2|u|^{\alpha+1}}du\Big]\\
=&\mathbb{E}\Big[\int_{|u|\geq(2A)^{\frac{1}{\alpha}}}\big[uf'(Y+au)-u{\bf 1}_{(-1,1)}(au)f'(Y)\big]\frac{A\alpha{k}_{\delta}(u)}
{|u|^{\alpha+1}}du\Big]+\mathcal{R},
\end{align*}
where
\begin{align}\label{R}
\mathcal{R}=\mathbb{E}\Big[\int_{-(2A)^{\frac{1}{\alpha}}}^{(2A)^{\frac{1}{\alpha}}}\big[uf'(Y+au)-u{\bf 1}_{(-1,1)}(au)f'(Y)\big]\frac{A\alpha{k}_{\delta}(u)}
{|u|^{\alpha+1}}du\Big].
\end{align}
Since $\int_{|u|\geq(2A)^{\frac{1}{\alpha}}}\frac{A\alpha{\bf 1}_{\delta}(u)}
{|u|^{\alpha+1}}du=1,$ we can consider a random variable $\tilde{X}$ which is independent of $Y$ and satisfies
\begin{align}\label{Fun}
\mathbb{P}(\tilde{X}>x)=\frac{A(1+\delta)}{|x|^{\alpha}},\quad x\geq(2A)^{\frac{1}{\alpha}},\qquad \mathbb{P}(\tilde{X}\leq x)=\frac{A(1-\delta)}{|x|^{\alpha}},\quad x\leq-(2A)^{\frac{1}{\alpha}},
\end{align}
it follows that
\begin{align*}
\frac{2A\alpha^{2}}{d_{\alpha}}a^{\alpha-1}\mathbb{E}\left[\mathcal{L}^{\alpha,\delta}f(Y)-\frac{\delta d_{\alpha}f'(Y)}{1-\alpha}\right]=\mathbb{E}\big[\tilde{X}f'(Y+a\tilde{X})\big]-\mathbb{E}[\tilde{X}{\bf 1}_{(-1,1)}(a\tilde{X})f'(Y)\big]
+\mathcal{R}.
\end{align*}
As a result, denoting by $F_{\tilde{X}}$ the distribution function of $\tilde{X}$,
we have
\begin{align}\label{result}
T_{1}\leq&\mathbb{E}\Big|\int_{-\infty}^{\infty}\big[xf'(Y+ax)-x{\bf 1}_{(-1,1)}(ax)f'(Y)\big]
d\big(F_{X}(x)-F_{\tilde{X}}(x)\big)\Big|+|\mathcal{R}|,
\end{align}
and it is easy to verify by (\ref{e:HolF'}),
\begin{align}\label{R1}
|\mathcal{R}|\leq 2A\alpha\int_{-(2A)^{\frac{1}{\alpha}}}^{(2A)^{\frac{1}{\alpha}}}
\frac{\mathbb{E}\big|f'(Y+au)-f'(Y)\big|}
{|u|^{\alpha}}du\leq C_{\alpha,A}a^{\alpha}.
\end{align}
For the first term, we have
\begin{align*}
&\mathbb{E}\Big|\int_{-\infty}^{\infty}\big[xf'(Y+ax)-x{\bf 1}_{(-1,1)}(ax)f'(Y)\big]
d\big(F_{X}(x)-F_{\tilde{X}}(x)\big)\Big|\\
\leq&\mathbb{E}\Big|\Big(\int_{-a^{-1}}^{a^{-1}}+\int_{|x|>a^{-1}}\Big)\big[xf'(Y+ax)-x{\bf 1}_{(-1,1)}(ax)f'(Y)\big]
d\big(F_{X}(x)-F_{\tilde{X}}(x)\big)\Big|:=\mathcal{I}+\mathcal{II}.
\end{align*}
According to (\ref{dis}) and (\ref{Fun}), we immediately obtain
\begin{align}\label{function}
F_{X}(x)-F_{\tilde{X}}(x)=&\big(\frac{1}{2}-\frac{A+\epsilon(x)}{|x|^{\alpha}}\big)(1+\delta){\bf 1}_{(0,(2A)^{\frac{1}{\alpha}})}(x)-\frac{\epsilon(x)}{|x|^{\alpha}}(1+\delta){\bf 1}_{((2A)^{\frac{1}{\alpha}},\infty)}(x)\nonumber\\
&+\big(\frac{A+\epsilon(x)}{|x|^{\alpha}}-\frac{1}{2}\big)(1-\delta){\bf 1}_{(-(2A)^{\frac{1}{\alpha}},0)}(x)+\frac{\epsilon(x)}{|x|^{\alpha}}(1-\delta){\bf 1}_{(-\infty,-(2A)^{\frac{1}{\alpha}})}(x).
\end{align}
On the one hand, we have by (\ref{e:HolF'}) and (\ref{function}),
\begin{align*}
\mathcal{I}&\leq\mathbb{E}\Big[\int_{-a^{-1}}^{a^{-1}}|x||f'(Y+ax)-f'(Y)|
\big|d\big(F_{X}(x)-F_{\tilde{X}}(x)\big)\big|\Big]\\
&\leq C_{\alpha}a^{\alpha}\int_{-a^{-1}}^{a^{-1}}|x|^{1+\alpha}\big|d\big(F_{X}(x)-F_{\tilde{X}}(x)\big)\big|\leq C_{\alpha,A}\big(a^{\alpha}+a^{\alpha}\int_{-a^{-1}}^{a^{-1}}|x|^{1+\alpha}\big|d\frac{\varepsilon(x)}{|x|^{\alpha}}\big|\big).
\end{align*}
On the other hand, noting that $\epsilon$ is $C^2$ and $x\epsilon'(x)\rightarrow0$ as $|x|\rightarrow\infty,$ we have by integration by parts that
\begin{align*}
&\mathbb{E}\Big|\int_{a^{-1}}^{\infty}xf'(Y+ax)d\frac{\epsilon(x)}{x^{\alpha}}\Big|\\
\leq&\mathbb{E}\Big|\int_{a^{-1}}^{\infty}xf'(Y+ax)\frac{x\epsilon'(x)}{x^{\alpha+1}}dx\Big|+\mathbb{E}\Big|\int_{a^{-1}}^{\infty}xf'(Y+ax)\frac{\alpha\epsilon(x)}{x^{\alpha+1}}dx\Big|\\
=&\mathbb{E}\Big|\int_{a^{-1}}^{\infty}f'(Y+ax)dx\int_{x}^{\infty}d\frac{t\epsilon'(t)}{t^{\alpha}}\Big|
+\mathbb{E}\Big|\int_{a^{-1}}^{\infty}f'(Y+ax)dx\int_{x}^{\infty}d\frac{\alpha\epsilon(t)}{t^{\alpha}}\Big|\\
=&\mathbb{E}\Big|\int_{a^{-1}}^{\infty}d\frac{t\epsilon'(t)}{t^{\alpha}}\int_{a^{-1}}^{t}f'(Y+ax)dx\Big|
+\mathbb{E}\Big|\int_{a^{-1}}^{\infty}d\frac{\alpha\epsilon(t)}{t^{\alpha}}\int_{a^{-1}}^{t}f'(Y+ax)dx\Big|\\
=
&a^{-1}\mathbb{E}\Big|\int_{a^{-1}}^{\infty}\big(f(Y+at)-f(Y+1)\big)d\frac{t\epsilon'(t)}{t^{\alpha}}\Big|
+a^{-1}\mathbb{E}\Big|\int_{a^{-1}}^{\infty}\big(f(Y+at)-f(Y+1)\big)d\frac{\alpha\epsilon(t)}{t^{\alpha}}\Big|,
\end{align*}
then, we have by (\ref{e:Holf})
\begin{align*}
\mathbb{E}\Big|\int_{a^{-1}}^{\infty}xf'(Y+ax)d\frac{\epsilon(x)}{x^{\alpha}}\Big|
\leq a^{\beta-1}\int_{a^{-1}}^{\infty}|t|^{\beta}\big|d\frac{t\epsilon'(t)}{t^{\alpha}}\big|+a^{\beta-1}\int_{a^{-1}}^{\infty}|t|^{\beta}\big|d\frac{\alpha\epsilon(t)}{t^{\alpha}}\big|.
\end{align*}
Using the same argument, we get that
\begin{align*}
\mathbb{E}\Big|\int_{-\infty}^{-a^{-1}}xf'(Y+ax)d\frac{\epsilon(x)}{|x|^{\alpha}}\Big|\leq a^{\beta-1}\int^{-a^{-1}}_{-\infty}|t|^{\beta}\big|d\frac{t\epsilon'(t)}{|t|^{\alpha}}\big|+a^{\beta-1}\int^{-a^{-1}}_{-\infty}|t|^{\beta}\big|d\frac{\alpha\epsilon(t)}{|t|^{\alpha}}\big|.
\end{align*}
These imply
\begin{align}\label{1}
\mathcal{II}\leq a^{\beta-1}\int_{|x|\geq a^{-1}}|x|^{\beta}\big|d\frac{x\epsilon'(x)}{|x|^{\alpha}}\big|+a^{\beta-1}\int_{|x|\geq a^{-1}}|x|^{\beta}\big|d\frac{\alpha\epsilon(x)}{|x|^{\alpha}}\big|,
\end{align}
the desired conclusion follows.
\end{proof}

$\bullet$ \underline{{\bf $\alpha=1$}} {\bf :}

\begin{lem}\label{ml3}
Consider $\alpha=1$ and $\delta=0.$ Let $X$ have a distribution $F_{X}$ with the form (\ref{dis}), $X$ and $Y$ are independent. For any $0<a\leq(2A)^{-1}\wedge1$ and $f$ is defined as above, denote
\begin{align*}
T_{2}\!:=\!\Big|\mathbb{E}\big[Xf'(Y+aX)\big]\!-\!\mathbb{E}[X{\bf 1}_{(-1,1)}(aX)]\mathbb{E}[f'(Y)]-\frac{2A}{d_{1}}\mathbb{E}[\mathcal{L}^{1,0}f(Y)]\Big|,
\end{align*}
then
\begin{align*}
T_{2}\leq C_{\beta,A}\Big[&a-a\log a+a\int_{-a^{-1}}^{a^{-1}}|x|^{2}\big(2-\log|ax|\big)\big|d\frac{\epsilon(x)}{|x|}\big|\\
&+a^{\beta-1}\int_{|x|\geq a^{-1}}|x|^{\beta}\big(\big|d\frac{x\epsilon'(x)}{|x|}\big|+\big|d\frac{\epsilon(x)}{|x|}\big|\big)\Big].
\end{align*}
\end{lem}
\begin{proof}
By the same argument as (\ref{result}), we have
\begin{align*}
T_{2}\leq&\mathbb{E}\Big|\int_{-\infty}^{\infty}\big[xf'(Y+ax)-x{\bf 1}_{(-1,1)}(ax)f'(Y)\big]
d\big(F_{X}(x)-F_{\tilde{X}}(x)\big)\Big|+|\mathcal{R}|,
\end{align*}
where $F_{\tilde{X}}$ and $\mathcal{R}$ are defined by (\ref{Fun}) and (\ref{R}) with $\alpha=1$ and $\delta=0,$ respectively. Moreover, by (\ref{e:holdf'c}), it is easy to verify\\
\begin{align}
|\mathcal{R}|\leq A\int_{-(2A)}^{2A}\frac{\big|f'(Y+au)-f'(Y)\big|}
{|u|}du&\leq C_{A}a\int_{-2A}^{2A}\frac{(2-\log|au|)|u|}{|u|}\dif u\nonumber\\
&\leq C_{A}(a-a\log a).
\end{align}

For the first term, we have
\begin{align*}
&\mathbb{E}\Big|\int_{-\infty}^{\infty}\big[xf'(Y+ax)-x{\bf 1}_{(-1,1)}(ax)f'(Y)\big]
d\big(F_{X}(x)-F_{\tilde{X}}(x)\big)\Big|\\
\leq&\mathbb{E}\Big|\Big(\int_{-a^{-1}}^{a^{-1}}+\int_{|x|>a^{-1}}\Big)\big[xf'(Y+ax)-x{\bf 1}_{(-1,1)}(ax)f'(Y)\big]
d\big(F_{X}(x)-F_{\tilde{X}}(x)\big)\Big|:=\mathcal{J}_{1}+\mathcal{J}_{2}.
\end{align*}
On the one hand, we have by (\ref{e:holdf'c}) and (\ref{function})
\begin{align*}
\mathcal{J}_{1}\leq&Ca\int_{-a^{-1}}^{a^{-1}}|x|^{2}\big(2-\log|ax|\big)\big|d\big(F_{X}(x)-F_{\tilde{X}}(x)\big)\big|\\
\leq&C\Big(a+a\int_{-a^{-1}}^{a^{-1}}|x|^{2}\big(2-\log|ax|\big)\big|d\frac{\epsilon(x)}{|x|}\big|\Big).
\end{align*}
On the other hand, by the same argument as the proof of (\ref{1}), we have
\begin{align*}
\mathcal{J}_{2}\leq a^{\beta-1}\int_{|x|\geq a^{-1}}|x|^{\beta}\big|d\frac{x\epsilon'(x)}{|x|}\big|+a^{\beta-1}\int_{|x|\geq a^{-1}}|x|^{\beta}\big|d\frac{\epsilon(x)}{|x|}\big|,
\end{align*}
the desired conclusion follows.
\end{proof}

\subsection{Truncation for random variable $X$}
Let $X$ have a distribution of the form (\ref{dis}), then it is obvious that $\mathbb{E}|X|^\al=\infty$ in the case $\alpha\in(0,1].$  Due to this, we need to truncate random variables
 Before giving the truncation Lemma, we first recall the \cite[Lemma 2.6]{C-X}, which will be used from time to time later.

\begin{lem} [{\cite[Lemma 2.6]{C-X}}] \label{expectation}
Let $X\geq0$ be a random variable and $t>0$, then
\begin{align*}
\mathbb{E}\big[X{\bf 1}_{(0,t)}(X)\big]
=\int_{0}^{t}\mathbb{P}(X>r)\dif r-t\mathbb{P}(X>t).
\end{align*}
\end{lem}

Now, we are in a position to give the truncation lemma.

\begin{lem}\label{ltrun}
Consider $\alpha\in(0,1]$ and when $\alpha=1$ we assume $\delta=0.$ Let $X$ have a distribution of the form (\ref{dis}) and $f$ be defined as above. Then for any $0<a<1$ and $z\in\mathbb{R},$ we have\\
1.) when $\alpha\in(0,1)$
\begin{align*}
\quad\mathbb{E}\Big[\big|\mathcal{L}^{\alpha,\delta}f(z)-\mathcal{L}^{\alpha,\delta}f(z+aX)\big|\Big]
\leq C_{\alpha,\beta,A,K}a^{\alpha}.
\end{align*}
2.) when $\alpha=1,$
\begin{align*}
\quad\mathbb{E}\Big[\big|\mathcal{L}^{1,0}f(z)-\mathcal{L}^{1,0}f(z+aX)\big|\Big]
\leq C_{\beta,A,K}\big(1+\log a^{-1}+(\log a)^{2}\big)a.
\end{align*}
\end{lem}
\begin{proof}
Observe
\begin{align*}
&\quad\mathbb{E}\Big[\big|\mathcal{L}^{\alpha,\delta}f(z)-\mathcal{L}^{\alpha,\delta}f(z+aX)\big|\Big]\\
&=\mathbb{E}\Big[\big|\mathcal{L}^{\alpha,\delta}f(z)-\mathcal{L}^{\alpha,\delta}f(z+aX)\big|\big[{\bf 1}_{(a^{-1},\infty)}(|X|)+{\bf 1}_{((0,a^{-1}))}(|X|)\big]\Big]:=\mathrm{I}+\mathrm{II}.
\end{align*}
When $\alpha\in(0,1),$ one can write by (\ref{e:FraFBou}) and (\ref{dis})
\begin{align*}
\mathrm{I}\leq C_{\alpha,\beta}\mathbb{P}\big(|X|>a^{-1}\big)\leq C_{\alpha,\beta}\big(A+\sup_{|x|\geq a^{-1}}|\epsilon(x)|\big)a^{\alpha}\leq C_{\alpha,\beta,A,K}a^{\alpha},
\end{align*}
whereas by (\ref{regularity1}) with $\gamma=\frac{1+\alpha}{2}\in(\alpha,1)$ and Lemma \ref{expectation}
\begin{align*}
\mathrm{II}\leq C_{\alpha,\gamma}a^{\gamma}\mathbb{E}[|X|^{\gamma}{\bf 1}_{((0,a^{-1}))}(|X|)]
&\leq C_{\alpha,\gamma}a^{\gamma}\int_0^{a^{-1}}\mathbb{P}(|X|>y)y^{\gamma-1}\dif y\\
&\leq C_{\alpha,\gamma}a^{\gamma}\int_0^{a^{-1}}\frac{2(A+K)}{y^{\alpha-\gamma+1}}\dif y\leq C_{\alpha,A,K}a^{\alpha}.
\end{align*}
When $\alpha=1$ and $\delta=0,$ one can write by (\ref{e:FraFBouc})
\begin{align*}
\mathrm{I}\leq C_{\beta}\mathbb{P}\big(|X|>a^{-1}\big)\leq C_{\beta}\big(A+\sup_{|x|\geq a^{-1}}|\epsilon(x)|\big)a\leq C_{\beta,A,K}a,
\end{align*}
whereas by (\ref{e:FraFHolc})
\begin{align*}
\mathrm{II}&\leq C_{\beta}a\mathbb{E}\big[|X|(1-\log|aX|){\bf 1}_{(0,a^{-1})}(|X|)\big].
\end{align*}
Further,
by integration by parts
\begin{align}\label{s2}
&\quad\mathbb{E}\big[|X|(1-\log|aX|){\bf 1}_{(0,a^{-1})}(|X|)\big]=\int_{0}^{a^{-1}}x(1-\log (ax))dF_{|X|}(x)\nonumber\\
&=\int_{0}^{a^{-1}}\int_0^x(-\log (ay))\dif ydF_{|X|}(x)=\int_{0}^{a^{-1}}(-\log (ay))\mathbb{P}\left(y<|X|\le a^{-1}\right)\dif y\nonumber\\
&\leq \int_{0}^{1}(-\log (ay))\dif y+\int_{1}^{a^{-1}}(-\log (ay))\frac{2(A+K)}{y}\dif y\nonumber\\
&=1-\log a+(A+K)(\log a)^2.
\end{align}
Hence, we have
\begin{align*}
\mathrm{II}&\leq C_{\beta,A,K}\big(1+\log a^{-1}+(\log a)^{2}\big)a,
\end{align*}
the desired conclusion follows.
\end{proof}

\begin{lem}\label{ltrun2}
Consider $\alpha=1$ and $\delta=0.$ Let $X$ have a distribution of the form (\ref{dis}) and $f$ be defined as above. Then for any $0<a<1$ and $z\in\mathbb{R},$ we have
\begin{align*}
\mathbb{E}\big|f'(z)-f'(z+aX)\big|\leq C_{\beta,A,K}\big(1+\log a^{-1}+(\log a)^{2}\big)a.
\end{align*}
\end{lem}
\begin{proof}
Observe
\begin{align*}
\mathbb{E}\big|f'(z)-f'(z+aX)\big|&=\mathbb{E}\Big[\big|f'(z)-f'(z+aX)\big|\big[{\bf 1}_{(a^{-1},\infty)}(|X|)+{\bf 1}_{((0,a^{-1}))}(|X|)\big]\Big]\\
&\leq2\mathbb{P}(|X|>a^{-1})+C_{\beta}a\mathbb{E}\big[|X|(2-\log|aX|){\bf 1}_{(0,a^{-1})}(|X|)\big]\\
&\leq C_{\beta}\Big[\big(A+\sup_{|x|\geq a^{-1}}|\epsilon(x)|\big)a+\mathbb{E}\big[|X|(2-\log|aX|){\bf 1}_{(0,a^{-1})}(|X|)\big]a\Big]\\
&\leq C_{\beta,A,K}\big(1+\log a^{-1}+(\log a)^{2}\big)a,
\end{align*}
where the first inequality thanks to (\ref{e:F'LeAlpc}) and (\ref{e:holdf'c}), the last inequality thanks to (\ref{s2}). The proof is complete.
\end{proof}
\subsection{Leave-one out method and proof of Theorem \ref{thm2}}
With the above results, we can extend the celebrated Stein's leave-one out approach of normal approximation (see \cite[pages 5-6]{ChGoSh11}) .

Recall the notation introduced in Theorem \ref{thm2}, we have $\sigma=(\frac{2A\alpha}{d_{\alpha}})^{\frac{1}{\alpha}}$ and let $S_{n,i}=S_{n}-\frac{n^{-\frac{1}{\alpha}}}{\sigma}X_{i}.$
By observing that $S_{n,i}$ and $X_{i}$ are independent, one can write
\begin{align*}
\Big|\mathbb{E}\big[S_{n}f'(S_{n})\big]-\alpha\mathbb{E}\big[(\mathcal{L}^{\alpha,\delta}f)(S_{n})\big]\Big|\leq \mathrm I+ \mathrm {II}+ \mathrm{III},
\end{align*}
where
\begin{align*}
\mathrm I=\frac{\alpha}{n}\sum_{i=1}^{n}\Big|\mathbb{E}\big[(\mathcal{L}^{\alpha,\delta}f)(S_{n,i})-\mathbb{E}\big[(\mathcal{L}^{\alpha,\delta}f)(S_{n})\big]\Big|,
\end{align*}
in the case $\alpha\in(0,1)$,
\begin{align*}
\mathrm{II}=\frac{n^{-\frac{1}{\alpha}}}{\sigma}\sum_{i=1}^{n}\Big|\mathbb{E}\big[&X_{i}f'(S_{n,i}+\frac{n^{-\frac{1}{\alpha}}}{\sigma}X_{i})\big]-\mathbb{E}\big[X_{i}{\bf 1}_{(0,\sigma n^{\frac{1}{\alpha}})}(|X_{i}|)\big]\mathbb{E}\big[f'(S_{n,i})\big]\\
&-\frac{2A\alpha^{2}}{d_{\alpha}}(\frac{n^{-\frac{1}{\alpha}}}{\sigma})^{\alpha-1}\mathbb{E}\big[(\mathcal{L}^{\alpha,\delta}f)(S_{n,i})-\frac{\delta d_{\alpha}}{\alpha(1-\alpha)}f'(S_{n,i})\big]\Big|,
\end{align*}
\begin{align*}
\mathrm{III}=\frac{n^{-\frac{1}{\alpha}}}{\sigma}\sum_{i=1}^{n}\Big|\frac{2A\alpha\delta}{1-\alpha}(\frac{n^{-\frac{1}{\alpha}}}{\sigma})^{\alpha-1}\mathbb{E}\big[f'(S_{n,i})\big]-\mathbb{E}\big[X_{i}{\bf 1}_{(0,\sigma n^{\frac{1}{\alpha}})}(|X_{i}|)\big]\mathbb{E}\big[f'(S_{n,i})\big]\Big|,
\end{align*}
and in the case $\alpha=1,$ $\delta=0,$
\begin{align*}
\mathrm{II}=\frac{n^{-1}}{\sigma}\sum_{i=1}^{n}\Big|\mathbb{E}\big[&X_{i}f'(S_{n,i}+\frac{n^{-1}}{\sigma}X_{i})\big]-\mathbb{E}\big[X_{i}{\bf 1}_{(0,\sigma n)}(|X_{i}|)f'(S_{n,i})\big]-\frac{2A}{d_{1}}\mathbb{E}\big[(\mathcal{L}^{1,0}f)(S_{n,i})\big]\Big|,
\end{align*}
\begin{align*}
\mathrm{III}=\frac{n^{-1}}{\sigma}\sum_{i=1}^{n}\big|\mathbb{E}\big[X_{i}{\bf 1}_{(0,\sigma n)}(|X_{i}|)\big]\big|\Big|\mathbb{E}\Big[f'(S_{n,i})-f'\big(S_{n,i}+\frac{n^{-1}}{\sigma}X_{i}\big)\Big]\Big|.
\end{align*}
1) When $\alpha\in(0,1),$ we have by Lemma \ref{ltrun},
\begin{align*}
I\leq&C_{\alpha,\beta,A,K}n^{-1}.
\end{align*}
By Lemma \ref{ml2}, we have
\begin{align*}
\mathrm{II}\leq C_{\alpha,\beta,A}\Big[n^{-\frac{1}{\alpha}}+n^{-\frac{1}{\alpha}}\int_{-\sigma n^{\frac{1}{\alpha}}}^{\sigma n^{\frac{1}{\alpha}}}|x|^{1+\alpha}\big|d\frac{\epsilon(x)}{|x|^{\alpha}}\big|+
n^{\frac{\alpha-\beta}{\alpha}}
\int_{|x|\geq \sigma n^{\frac{1}{\alpha}}}|x|^{\beta}\big(\big|d\frac{x\epsilon'(x)}{|x|^{\alpha}}\big|+\big|d\frac{\epsilon(x)}{|x|^{\alpha}}\big|\big)\Big].
\end{align*}
In addition, we have by Lemma \ref{expectation}
\begin{align*}
\mathbb{E}\Big[X_{i}{\bf 1}_{(0,\sigma n^{\frac{1}{\alpha}})}(|X_{i}|)\Big]=\frac{2A\alpha\delta}{1-\alpha}\big(\frac{n^{-\frac{1}{\alpha}}}{\sigma}\big)^{\alpha-1}&+(1+\delta)\int_{0}^{\sigma n^{\frac{1}{\alpha}}}\frac{\epsilon(x)}{x^{\alpha}}\dif x-(1-\delta)\int_{0}^{\sigma n^{\frac{1}{\alpha}}}\frac{\epsilon(-x)}{x^{\alpha}}\dif x\\
&+\big(\frac{n^{-\frac{1}{\alpha}}}{\sigma}\big)^{\alpha-1}\big[(1-\delta)\epsilon(-\sigma n^{\frac{1}{\alpha}})-(1+\delta)\epsilon(\sigma n^{\frac{1}{\alpha}})\big],
\end{align*}
which follows that
\begin{align*}
\mathrm{III}\leq C_{\alpha}\Big[\sup_{|x|\geq\sigma n^{\frac{1}{\alpha}}}|\epsilon(x)|+n^{\frac{\alpha-1}{\alpha}}\Big|(1+\delta)\int_{0}^{\sigma n^{\frac{1}{\alpha}}}\frac{\epsilon(x)}{x^{\alpha}}\dif x-(1-\delta)\int_{0}^{\sigma n^{\frac{1}{\alpha}}}\frac{\epsilon(-x)}{x^{\alpha}}\dif x\Big|\Big].
\end{align*}
2) When $\alpha=1$ and $\delta=0,$ we have by Lemma \ref{ltrun},
\begin{align*}
\mathrm I\leq&C_{\beta,A,K}\big(1+\log n+(\log n)^{2}\big)n^{-1}.
\end{align*}
By Lemma \ref{ml3}, we have
\begin{align*}
\mathrm{II}\leq C_{\beta,A}\Big[&n^{-1}(1+\log n)+n^{-1}\int_{-\sigma n}^{\sigma n}|x|^{2}\big(2-\log|\frac{n^{-1}}{\sigma}x|\big)\big|d\frac{\epsilon(x)}{|x|}\big|\\
&+n^{1-\beta}
\int_{|x|\geq \sigma n}|x|^{\beta}\big(\big|d\frac{x\epsilon'(x)}{|x|}\big|+\big|d\frac{\epsilon(x)}{|x|}\big|\big)\Big].
\end{align*}
In addition, we have by Lemma \ref{expectation}
\begin{align*}
\Big|\mathbb{E}\Big[X_{i}{\bf 1}_{(0,\sigma n)}(|X_{i}|)\Big]\Big|&=\Big|\int_{0}^{\sigma n}\frac{\epsilon(x)-\epsilon(-x)}{x}\dif x+\epsilon(-\sigma n)-\epsilon(\sigma n)\Big|\\
&\leq \Big|\int_{0}^{\sigma n}\frac{\epsilon(x)-\epsilon(-x)}{x}\dif x\Big|+2K,
\end{align*}
which follows from Lemma \ref{ltrun2} that
\begin{align*}
\mathrm{III}\leq C_{\beta,A,K}\big(1+\log n+(\log n)^{2}\big)\Big(n^{-1}+n^{-1}\Big|\int_{0}^{\sigma n}\frac{\epsilon(x)-\epsilon(-x)}{x}\dif x\Big|\Big).
\end{align*}
Combining all of above, the desired conclusion follows.
\qed

\section{Three Examples}\label{s:examples}

\medskip
{\it \underline{Example 1}: Pareto distribution case \cite{DaNa02,KuKe00}}.
Our first example is the simplest situation, that is, the case where $X_1$ is distributed according to a (possibly non-symmetric)
Pareto distribution of the form
\begin{align*}
\mathbb{P}(X_{1}> x)=\frac{1+\delta}{2|x|^{\alpha}},\quad x\geq1,\qquad \mathbb{P}(X_{1}\leq x)=\frac{1-\delta}{2|x|^{\alpha}},\quad x\leq-1,
\end{align*}
with $\alpha\in(0,1]$ and $\delta\in[-1,1]$.

In this case, (\ref{dis}) holds with $A=\frac{1}{2},$
$\epsilon(x)=\frac{|x|^\alpha-1}{2}{\bf 1}_{(-1,1)}(x)$
and $K=\frac{1}{2}.$
Clearly, we see that $L=1$, $\epsilon'(x)=0$ and
\begin{align*}
 F_X(x)=F_{\tilde{X}}(x).
\end{align*}
According to Theorem \ref{thm2},

\noindent
1) When $\alpha\in(0,1),$ we have
\begin{align*}
\quad \dwb\big(S_n,Z\big)
\leq C_{\alpha,\beta,\delta}(n^{-1}+\frac{1-\alpha}{\alpha}\delta n^{\frac{\alpha-1}{\alpha}})
\end{align*}
where $Z\sim S_\al(\de)$. In particular, when $\delta=0,$ we immediately obtain $\dwb\big(S_n,Z\big)=O(n^{-1}).$
\smallskip

\noindent
2) When $\alpha=1$ and $\delta=0,$ we have
\begin{align*}
\quad \dwb\big(S_n,Z\big)\leq C_{\beta}n^{-1}(\log n)^{2}
\end{align*}
where $Z\sim S_1(0)$.
\qed
\medskip

{\it \underline{Example 2}: Heavy tail with mixed decay rate \cite{KuKe00}}. We consider
\begin{equation}\label{e:5.1}
\begin{cases}
\mathbb{P}(X_{1}> x)=(A|x|^{-\alpha}+\widetilde{A}|x|^{-\widetilde{\alpha}})(1+\delta),\quad x\geq 1,\\
\mathbb{P}(X_{1}\leq x)=(A|x|^{-\alpha}+\widetilde{A}|x|^{-\widetilde{\alpha}})(1-\delta),\quad x\leq-1,
\end{cases}
\end{equation}
with $\alpha\in(0,1]$, $\alpha<\widetilde{\alpha}$, $A+\widetilde{A}=\frac12$ and $\delta\in[-1,1]$.

In this case, (\ref{dis}) holds with
\[
\epsilon(x)=\widetilde{A}|x|^{\alpha-\widetilde{\alpha}}{\bf 1}_{[1,\infty)}(|x|)+\Big(\frac{|x|^\alpha}{2}-A\Big){\bf 1}_{(-1,1)}(x)
\]
and $K=A\vee\tilde{A}.$ Since $x\epsilon'(x)=\tilde{A}(\alpha-\tilde{\alpha})|x|^{\alpha-\tilde{\alpha}}$ for any $|x|>1,$ we have by Theorem \ref{thm2} with $L=1$\\
\noindent
1) When $\alpha\in(0,1),$ we have
\begin{align*}
&n^{-\frac{1}{\alpha}}\int_{-\sigma n^{\frac{1}{\alpha}}}^{\sigma n^{\frac{1}{\alpha}}}|x|^{1+\alpha}\big|d\frac{\epsilon(x)}{|x|^{\alpha}}\big|\\
\leq& C_{\alpha,K,\tilde{A}}\big(n^{-\frac{1}{\alpha}}+n^{\frac{\alpha-\tilde{\alpha}}{\alpha}}\big)+C_{\alpha,\tilde{A}}n^{-\frac{1}{\alpha}}
\begin{cases}
\frac{1}{\tilde{\alpha}-\alpha-1}, &\tilde{\alpha}>1+\alpha,\\
\log(\sigma n^{\frac{1}{\alpha}}), &\tilde{\alpha}=1+\alpha,\\
\frac{\sigma^{1+\alpha-\tilde{\alpha}}}{1+\alpha-\tilde{\alpha}}n^{\frac{1+\alpha-\tilde{\alpha}}{\alpha}}, &\alpha<\tilde{\alpha}<1+\alpha,
\end{cases}
\end{align*}
\begin{align*}
n^{\frac{\alpha-\beta}{\alpha}}
\int_{|x|\geq \sigma n^{\frac{1}{\alpha}}}|x|^{\beta}\big(\big|d\frac{x\epsilon'(x)}{|x|^{\alpha}}\big|+\big|d\frac{\epsilon(x)}{|x|^{\alpha}}\big|\big)\leq C_{\alpha,\tilde{\alpha},\tilde{A}}n^{\frac{\alpha-\tilde{\alpha}}{\alpha}},
\end{align*}
and
\begin{align}\label{symmetric}
\mathcal{R}_{1,n}\leq& C_{\alpha,\tilde{\alpha},A,\tilde{A}}\big(n^{\frac{\alpha-\tilde{\alpha}}{\alpha}}+\delta n^{\frac{\alpha-1}{\alpha}}\big)+2\tilde{A}\delta n^{\frac{\alpha-1}{\alpha}}
\begin{cases}
\frac{1}{\tilde{\alpha}-1}, &\tilde{\alpha}>1,\\
\log(\sigma n^{\frac{1}{\alpha}}), &\tilde{\alpha}=1,\\
\frac{\sigma^{1-\tilde{\alpha}}}{1-\tilde{\alpha}}n^{\frac{1-\tilde{\alpha}}{\alpha}}, &\tilde{\alpha}<1.
\end{cases}
\end{align}
Combining all of above, we have when $\tilde{\alpha}\neq1$
\begin{align*}
\dwb\big(S_n,Z\big)=O\big(n^{-1}+n^{\frac{\alpha-\tilde{\alpha}}{\alpha}}+\delta n^{\frac{\alpha-1}{\alpha}}\big),
\end{align*}
and when $\tilde{\alpha}=1,$
\begin{align*}
\dwb\big(S_n,Z\big)=O\big(n^{-1}+n^{\frac{\alpha-1}{\alpha}}+\delta n^{\frac{\alpha-1}{\alpha}}\log n\big),
\end{align*}
where $Z\sim S_\al(\de)$.

\begin{rem}
In (\ref{symmetric}), when $\delta=0,$ we immediately obtain
\begin{align*}
\dwb\big(S_n,Z\big)=O\big(n^{-1}+n^{\frac{\alpha-\tilde{\alpha}}{\alpha}}\big).
\end{align*}
\end{rem}

\noindent
2) When $\alpha=1$ and $\delta=0,$ we have
\begin{align*}
&n^{-1}\int_{-\sigma n}^{\sigma n}|x|^{2}\big(2-\log|\frac{n^{-1}}{\sigma}x|\big)\big|d\frac{\epsilon(x)}{|x|}\big|\\
\leq&C_{K,\tilde{A}}\big(n^{-1}+n^{-1}\log n+n^{1-\tilde{\alpha}}\big)+C_{\tilde{A}}n^{-1}
\begin{cases}
\frac{1}{\tilde{\alpha}-2}, &\tilde{\alpha}>2,\\
\log(\sigma n), &\tilde{\alpha}=2,\\
\frac{\sigma^{2-\tilde{\alpha}}}{2-\tilde{\alpha}}n^{2-\tilde{\alpha}}, &1<\tilde{\alpha}<2,
\end{cases}
\end{align*}
\begin{align*}
n^{1-\beta}
\int_{|x|\geq \sigma n}|x|^{\beta}\big(\big|d\frac{x\epsilon'(x)}{|x|}\big|+\big|d\frac{\epsilon(x)}{|x|}\big|\big)\leq  C_{\tilde{\alpha},\tilde{A}}n^{1-\tilde{\alpha}}.
\end{align*}
and
\begin{align*}
\mathcal{R}_{2,n}=0.
\end{align*}
Hence, we have when $\tilde{\alpha}\geq2,$
\begin{align*}
\dwb\big(S_n,Z\big)=O\big(n^{-1}(\log n)^{2}\big),
\end{align*}
and when $\tilde{\alpha}\in(1,2),$ we have
\begin{align*}
\dwb\big(S_n,Z\big)=O\big(n^{1-\tilde{\alpha}}\big).
\end{align*}
\qed
\medskip

{\it \underline{Example 3}: Cauchy approximation for the reciprocal of independent sums \cite{Dia}.  It is known that for $Z$ having the  Cauchy distribution $S_1(0)$, the following distributional identity holds:
\begin{align*}
Z\overset{d}= \frac 1	Z.
\end{align*}
We consider
\begin{align*}
T_{n}=\frac{\sigma n}{X_{1}+\cdots+X_{n}-n\mathbb{E}\big[X_{1}{\bf 1}_{(0,\sigma n)}(|X_{1}|)\big]},
\end{align*}
where $X_{1},\cdots,X_{n}$ are independent and identically distributed random variables and $X_{1}$ has a distribution of the form (\ref{dis}) with $\al=1, \de=0$.

Since
\begin{align*}
\mathbb{P}(T_{n}\leq x)-\mathbb{P}(Z\leq x)=
\begin{cases}
\mathbb{P}(Z\leq\frac{1}{x})-\mathbb{P}(S_{n}\leq\frac{1}{x}),\quad &x\neq0,\\
\mathbb{P}(S_{n}\leq x)-\mathbb{P}(Z\leq x),\quad &x=0,
\end{cases}
\end{align*}
we have by (\ref{kol}) that
\begin{align*}
\dko(T_n,Z)=\dko(S_n,Z)\leq \!C_{\alpha}\!\Big[\dwb\big(S_n,Z\big)\Big]^{\frac{1}{2}}.
\end{align*}
Therefore, $\dwb\big(S_n, Z\big)\rightarrow0$ implies $ \dko(S_n,Z)\to 0$.
In particular, if $X_1$ has the Pareto law as in Example 1, we have
\begin{align*}
\dko(T_n,Y)\leq C_{\alpha,\beta}n^{-\frac{1}{2}}\log n.
\end{align*}
If $X_1$ has the tail of the form \eqref{e:5.1} as in Example 2, we have
\begin{align*}
\dko(T_n,Y)=
\begin{cases}
O\big(n^{-\frac{1}{2}}\log n\big),\quad &\tilde{\alpha}\geq2,\\
O\big(n^{\frac{1-\tilde{\alpha}}{2}}\big),\quad &\tilde{\alpha}\in(1,2).
\end{cases}
\end{align*}
\qed
\medskip
\appendix

\section{Auxiliary results}


\subsection{Extending a result of Albevio et al.}\label{a:albevio}
Suppose that  $\E\mathcal A^{\al,\delta}f(Y)=0$ for all $f$ in the Schwarz space of smooth fast decaying functions. Let $\mathcal A$ be an operator defined by $\mathcal A f(x)=-x f'(x)+2 f''(x)$ for all $f$ in Schwarz space, $\mathcal A$ generates an OU process with an ergodic measure $N(0,1)$. Let $N$ be an $N(0,1)$-distributed random variable, then $\E \mathcal A f(N)=0$ and $Y+N$ has a density function $\rho$. Moreover,
$$\E\left[(\mathcal A^{\al,\delta}+\mathcal A) f(Y+N)\right]=\E\left[\mathcal A^{\al,\delta} f(Y+N)\right]+\E\left[\mathcal A f(Y+N)\right]=0,$$
Applying Parseval identity (that involves Fourier transform in the sense of distribution) as in the proof of \cite[Prop. 3.1]{ARW} to $\E\left[(\mathcal A^{\al,\delta}+\mathcal A) f(Y+N)\right]=0$, we obtain  the following:
\begin{align*}
(\log \E[e^{i\lambda (Z+N)}]) \widehat{\rho}(\lambda) =  i\lambda \frac{d}{d\lambda}\widehat{\rho}(\lambda)  \quad \mbox{a.e.},
\end{align*}
where $\hat \rho$ is is Fourier transform of $\rho$.
We see that the unique density function such that the above differential equation holds is that of $Z+N$, ending the proof.
\qed

\subsection{Stationarity of increments of the process $(V_t)_{t\ge 0}$}\label{a:V_si}

By \cite[Eq. (17.3)]{sato}, one has
\begin{align*}
\E[e^{i\lambda(V_t-V_s)}] = \exp \left[\int_{\log(1+s)}^{\log(1+t)} \psi(e^{\frac{u}{\al}}\lambda) du \right] = \exp[(t-s)\psi(\lambda)],
\end{align*}
where $\psi(\lambda)=\log \E[e^{i\lambda Z}]$, as desired.
\qed

\subsection{Solution to Stein's equation}\label{a:solve_Stein}

Before proving Lemma \ref{lm23}, we first give the following lemma:
\begin{lem}\label{A1}
Let $(Q_{t})_{t\geq0}$ be a Markovian semigroup with transition density  $q(t,x,y)=p(1-e^{-t},y-e^{-\frac{t}{\alpha}}x)$.  Then we have
\begin{align}
\partial_{t} Q_{t}h(x)=\mathcal{A}_{\al,\de}Q_{t}h(x).
\end{align}
for any $h\in\mathcal H_\be$ with $0<\be<\al$.
\end{lem}
\begin{proof}
Recall that $q(t,x,y)=p(1-e^{-t},y-e^{-\frac{t}{\alpha}}x).$ Then
\begin{align*}
  \left|\frac{\partial}{\partial t}q(t,x,y)\right| &= \left|e^{-t}\frac{\partial p}{\partial t}(1-e^{-t},y-e^{-\frac{t}{\alpha}}x)+\al^{-1}e^{-\frac{t}{\alpha}}\frac{\partial p}{\partial x}(1-e^{-t},y-e^{-\frac{t}{\alpha}}x)\right|\\
   & \le \frac{C_\alpha}{((1-e^{-t})^{1/\al}+|y-e^{-\frac{t}{\alpha}}x|)^{1+\al}}
   +\frac{1}{\al}e^{-\frac{t}{\alpha}}\frac{C_\alpha(1-e^{-t})^{(\al-1)/\al}}{((1-e^{-t})^{1/\al}+|y-e^{-\frac{t}{\alpha}}x|)^{1+\al}}\\
   &\le \frac{C_\al(1+(e^t-1)^{-1/\al})}{((1-e^{-t})^{1/\al}+|y-e^{-\frac{t}{\alpha}}x|)^{1+\al}}
\end{align*}
where the second inequality above follows from $\frac{\partial p}{\partial t}(t,x)=\mathcal{L}^{\al,\de}p(t,x)$ and \cite[(3.2)]{C-Z}.
Thus, for $t>0$, $s>0$ small enough such that $(1-e^{-s/\al})|x|\le \frac{1}{2}(e^t-1)^{1/\al},$
\begin{align*}
  |q(t+s,x,y)-q(t,x,y)| & \le s \frac{C_\al 2^{1+\al}(1+(e^t-1)^{-1/\al})}{((1-e^{-t})^{1/\al}+|y-e^{-\frac{t}{\alpha}}x|)^{1+\al}}.
\end{align*}

In addition, according to (\ref{e:operator_OU}) and (\ref{e:OU}), we have
\begin{align}\label{heat}
\partial_{t}q(t,x,y)=\mathcal{A}_{\al,\de}q(t,x,y).
\end{align}
Hence, using dominated convergence theorem, (\ref{heat}) and Fubini's theorem, we have
\begin{align*}
\partial_{t} Q_{t}h(x)&=\partial_{t}\int_{\mathbb{R}}q(t,x,y)h(y)dy=\int_{\mathbb{R}}\partial_{t}q(t,x,y)h(y)dy\\
&=\int_{\mathbb{R}}\mathcal{A}_{\al,\de}q(t,x,y)h(y)dy=\mathcal{A}_{\al,\de}\int_{\mathbb{R}}q(t,x,y)h(y)dy=\mathcal{A}_{\al,\de}Q_{t}h(x),
\end{align*}
the desired conclusion follows.
\end{proof}
\noindent
{\it Proof of Lemma \ref{lm23}}.
\begin{align}\label{rep}
Q_{t}h(x)=\int_{\mathbb{R}}p(1-e^{-t},y-e^{-\frac{t}{\alpha}}x)h(y)dy=\int_{\mathbb{R}}p(1,y)h\big((1-e^{-t})^{\frac{1}{\alpha}}y+e^{-\frac{t}{\alpha}}x\big)dy,
\end{align}
and
\begin{align*}
\Big|h\big((1-e^{-t})^{\frac{1}{\alpha}}y+e^{-\frac{t}{\alpha}}(x+z)\big)-h\big((1-e^{-t})^{\frac{1}{\alpha}}y+e^{-\frac{t}{\alpha}}x\big)\Big|\leq e^{-\frac{\beta t}{\alpha}}(|z|\wedge|z|^{\beta}).
\end{align*}
By (\ref{rep}), we immediately have
\begin{align}\label{ineq}
\Big|Q_{t}h(x+z)-Q_{t}h(x)\Big|\leq \int_{\mathbb{R}}p(1,y)e^{-\frac{\beta t}{\alpha}}(|z|\wedge|z|^{\beta})dy=e^{-\frac{\beta t}{\alpha}}(|z|\wedge|z|^{\beta}).
\end{align} Recall $\mathcal{A}_{\al,\de}f(x)=\mathcal{L}^{\alpha,\delta}f(x)-\frac{1}{\alpha}xf'(x).$
By \eqref{ineq}, using the dominated convergence theorem, we get that
$$f'(x)=-\int_0^\infty \frac{\partial }{\partial x}Q_th(x)\,dt.$$
When $\alpha\in(0,1),$ we have
\begin{align*}
\mathcal{L}^{\alpha,\delta}f(x)=-\int_{-\infty}^{\infty}\int_{0}^{\infty}\big(Q_{t}h(x+z)-Q_{t}h(x)\big)dt\nu_{\alpha,\delta}(dz)
\end{align*}
and we have by (\ref{ineq})
\begin{align*}
\int_{-\infty}^{\infty}\int_{0}^{\infty}\big|Q_{t}h(x+z)-Q_{t}h(x)\big|dt\nu_{\alpha,\delta}(dz)&\leq C_{\alpha}\int_{-\infty}^{\infty}\int_{0}^{\infty}\frac{|Q_{t}h(x+z)-Q_{t}h(x)|}{|z|^{\alpha+1}}dtdz\\
&\leq C_{\alpha}\int_{0}^{\infty}e^{-\frac{\beta t}{\alpha}}dt\int_{-\infty}^{\infty}\frac{|z|\wedge|z|^{\beta}}{|z|^{\alpha+1}}\leq C_{\alpha,\beta}.
\end{align*}
When $\alpha=1$ and $\delta=0,$ we have
\begin{align*}
\mathcal{L}^{\alpha,\delta}f(x)&=\int_{-\infty}^{\infty}\int_{0}^{\infty}\big(Q_{t}h(x+z)-Q_{t}h(x)-z{\bf 1}_{(-1,1)}(z)\frac{d}{dx}Q_{t}h(x)\big)dt\nu_{\alpha,\delta}(dz)\\
&=\int_{-1}^{1}\int_{0}^{\infty}\int_{0}^{1}z\big(\frac{d}{dx}Q_{t}h(x+zs)-\frac{d}{dx}Q_{t}h(x)\big)dsdt\nu_{\alpha,\delta}(dz)\\
&\quad+\int_{|z|\geq1}\int_{0}^{\infty}\big(Q_{t}h(x+z)-Q_{t}h(x)\big)dt\nu_{\alpha,\delta}(dz)
\end{align*}
and by integration by parts,
\begin{align*}
&\Big|\frac{d}{dx}Q_{t}h(x+zs)-\frac{d}{dx}Q_{t}h(x)\Big|\\
\leq& e^{-t}\int_{\mathbb{R}}\Big|p\big(y-(1-e^{-t})^{-1}e^{-t}(x+zs)\big)-p\big(y-(1-e^{-t})^{-1}e^{-t}x\big)\Big|\big|h'\big((1-e^{-t})y\big)\big|dy\\
\leq&C e^{-t}\big((1-e^{-t})^{-1}e^{-t}|zs|\wedge1\big),
\end{align*}
these imply
\begin{align*}
&\int_{-\infty}^{\infty}\int_{0}^{\infty}\big|Q_{t}h(x+z)-Q_{t}h(x)-z{\bf 1}_{(-1,1)}(z)\frac{d}{dx}Q_{t}h(x)\big|dt\nu_{\alpha,\delta}(dz)\\
\leq& C\int_{-1}^{1}\int_{0}^{\infty}\frac{e^{-t}\big((1-e^{-t})^{-1}e^{-t}|z|\wedge1\big)}{|z|}dtdz+\int_{|z|\geq1}\int_{0}^{\infty}\frac{e^{-\frac{\beta t}{\alpha}}|z|^{\beta}}{|z|^{2}}dtdz\leq C_{\beta}.
\end{align*}
Thus, by Fubini's theorem, we get that, for $\al\in(0,1]$,
\begin{align*}
\mathcal{L}^{\alpha,\delta}f(x)=-\int_{0}^{\infty}\mathcal{L}^{\alpha,\delta}Q_{t}h(x)dt.
\end{align*}
Hence, according to Lemma \ref{A1}, we can obtain
\begin{align*}
\mathcal{A_{\alpha,\delta}}f=-\int_{0}^{\infty}\mathcal{A}_{\alpha,\delta}Q_{t}hdt=-\int_{0}^{\infty}\partial_{t}Q_{t}hdt=Q_{0}h-Q_{\infty}h,
\end{align*}
here $Q_{\infty}=\mu,$ the unique invariant distribution of the semigroup $(Q_t)_{t\ge 0}$ associated with $\mathcal A_{\alpha,\delta}$ by \cite[Cor. 17.9]{sato}. The proof is complete.
\section{Example 4: Regularly varying tails \cite{JuPa98}}\label{s:rv}
Assume that $X_1,X_2,\cdots$ be i.i.d. with a common density function $p_X(x)$
\[
p_X(x)=\frac{\alpha^2e^\alpha}{2(1+\alpha)}\,\frac{\log |x|}{|x|^{\alpha+1}}{\bf 1}_{[e,\infty)}(|x|), \quad\mbox{with $\alpha\in(0,1]$}.
\]
Then, we can prove
\begin{equation}\label{example}
\dwb\left(\frac1{\sigma\,\gamma_n}(X_1+\ldots+X_n),Z\right)=O\big((\log n)^{-1}\big).
\end{equation}

According to \cite[Theorem 3.7.2]{Dur10}, we define the sequence $(\gamma_n)_{n\geq 1}$ implicitly by $\gamma_n=\big(n\log \gamma_n\big)^\frac1\alpha$.
Observe that $X_1$ is integrable and centered.
We set $\sigma =\left(\frac{\alpha^2e^\alpha}{(1+\alpha)d_\alpha}\right)^{\frac1\alpha}$, $\widetilde{X}_i=\frac{n^\frac1\alpha}{\gamma_n}X_i$, $\widetilde{S}_n=\frac1\sigma\,n^{-\frac1\alpha}(\widetilde{X}_1+\ldots+\widetilde{X}_n)$, and
$\widetilde{S}_{n,i}=\widetilde{S}_n-\frac{n^{-\frac1\alpha}}{\sigma}\widetilde{X}_i$.

To prove (\ref{example}), we shall use Theorem \ref{thm1} with $\delta=0$.
Let $f$ be defined as above. We can write
\begin{eqnarray*}
&&\E[\widetilde{S}_nf'(\widetilde{S}_n)]-\alpha\,\E\big[
(\mathcal{L}^{\alpha,0}f)(\widetilde{S}_n)\big]\\
&=&\frac\alpha n \sum_{i=1}^n \left(
\E\big[
(\mathcal{L}^{\alpha,0}f)(\widetilde{S}_{n,i})\big]
-\E\big[
(\mathcal{L}^{\alpha,0}f)(\widetilde{S}_n)\big]
\right)\\
&&+\frac{n^{-\frac1\alpha}}{\sigma}\sum_{i=1}^n
\left(
\E\left[
\widetilde{X}_i\,f'(\widetilde{S}_{n,i}+\frac{n^{-\frac1\alpha}}{\sigma}\widetilde{X}_i)
\right]
-\frac{\alpha^3e^\alpha}{d_\alpha(1+\alpha)}\sigma^{1-\alpha}n^{\frac1\alpha-1}\,\E\big[
(\mathcal{L}^{\alpha,0}f)(\widetilde{S}_{n,i})\big]
\right).
\end{eqnarray*}
We have, using among other that $n\gamma_n^{-\alpha}=\frac1{\log \gamma_n}$,
\begin{eqnarray*}
&&\E\left[
\widetilde{X}_i\,f'(\widetilde{S}_{n,i}+\frac{n^{-\frac1\alpha}}{\sigma}\widetilde{X}_i)\right]\\
&=&\frac{\alpha^2e^\alpha}{2(1+\alpha)}\,\E\left[\int_\R \big(f'(\widetilde{S}_{n,i}+\frac{n^{-\frac1\alpha}}{\sigma}u)-f'(\widetilde{S}_{n,i})\big)
\frac{
u\log(
n^{-\frac1\alpha} \gamma_n\big|u|
)
}
{
|u|^{\alpha +1}\log \gamma_n
}\,
{\bf 1}_{[e,\infty)}(n^{-\frac1\alpha}\gamma_n\,|u|)\,
du\right]\\
&=&\frac{\alpha^2e^\alpha}{2(1+\alpha)}\,\E\left[\int_\R \big(f'(\widetilde{S}_{n,i}+\frac{n^{-\frac1\alpha}}{\sigma}u)-\phi'(\widetilde{S}_{n,i})\big)
\frac{
u
}
{
|u|^{\alpha +1}
}\,
{\bf 1}_{[e,\infty)}(n^{-\frac1\alpha}\gamma_n\,|u|)\,
du\right]\\
&&+\frac{\alpha^2e^\alpha}{2(1+\alpha)}\,\E\left[\int_\R \big(f'(\widetilde{S}_{n,i}+\frac{n^{-\frac1\alpha}}{\sigma}u)-f'(\widetilde{S}_{n,i})\big)
\frac{
u\log(
n^{-\frac1\alpha} \big|u|
)
}
{
|u|^{\alpha +1}\log \gamma_n
}\,
{\bf 1}_{[e,\infty)}(n^{-\frac1\alpha}\gamma_n\,|u|)\,
du\right].
\end{eqnarray*}
On the other hand, the Proposition \ref{properties} with $a=\frac{n^{-\frac1\alpha}}{\sigma}$ yields
\begin{eqnarray*}
\frac{2\alpha}{d_\alpha}\,\sigma^{1-\alpha}n^{\frac1\alpha-1}(\mathcal{L}^{\alpha,0}f)(\widetilde{S}_{n,i})
&=&\int_\R
\big(f'(\widetilde{S}_{n,i}+\frac{n^{-\frac1\alpha}}{\sigma}u)-f'(\widetilde{S}_{n,i})\big)\frac{u}{|u|^{\alpha +1}}du.
\end{eqnarray*}
As a result,
\begin{eqnarray*}
&&\left|\E\left[
\widetilde{X}_i\,f'(\widetilde{S}_{n,i}+\frac{n^{-\frac1\alpha}}{\sigma}\widetilde{X}_i)\right]- \frac{\alpha^3e^\alpha}{d_\alpha(1+\alpha)}\,\sigma^{1-\alpha}\,n^{\frac1\alpha-1}\,\E\left[(\mathcal{L}^{\alpha,0}f)(\widetilde{S}_{n,i})\right]\right|\\
&=&\left|\frac{\alpha^2e^\alpha}{2(1+\alpha)}\,\E\left[\int_\R \big(f'(\widetilde{S}_{n,i}+\frac{n^{-\frac1\alpha}}{\sigma}u)-f'(\widetilde{S}_{n,i})\big)
\frac{
u
}
{
|u|^{\alpha +1}
}\,
{\bf 1}_{(0,e)}(n^{-\frac1\alpha}\gamma_n\,|u|)\,
du\right]\right.\\
&&\left.+\frac{\alpha^2e^\alpha}{2(1+\alpha)}\,\E\left[\int_\R \big(f'(\widetilde{S}_{n,i}+\frac{n^{-\frac1\alpha}}{\sigma}u)-f'(\widetilde{S}_{n,i})\big)
\frac{
u\log(
n^{-\frac1\alpha} \big|u|
)
}
{
|u|^{\alpha +1}\log \gamma_n
}\,
{\bf 1}_{[e,\infty)}(n^{-\frac1\alpha}\gamma_n\,|u|)\,
du\right]\right|\\
&:=&\mathbf{I}+\mathbf{II}.
\end{eqnarray*}
Then, when $\alpha\in(0,1),$ we have by (\ref{e:HolF'})
\begin{align*}
\mathbf{I}\leq C_{\alpha}n^{\frac{1}{\alpha}-1}\gamma_{n}^{-1}.
\end{align*}
By (\ref{e:HolF'}), (\ref{stablestein}) and (\ref{e:FraFBou}), we have
\begin{align*}
\mathbf{II}\leq&C_{\alpha,\beta}\Big(\frac{n^{-1}}{\log\gamma_{n}}\int_{\frac{en^{\frac{1}{\alpha}}}{\gamma_{n}}}^{n^{\frac{1}{\alpha}}}|\log(n^{-\frac{1}{\alpha}}u)|\dif u+\frac{n^{\frac{1}{\alpha}}}{\log\gamma_{n}}\int_{n^{\frac{1}{\alpha}}}^{\infty}
\frac{1+n^{-\frac{\beta}{\alpha}}|u|^{\beta}}{|u|^{\alpha}}\dif u\Big)\\
\leq&C_{\alpha,\beta}\frac{n^{\frac{1}{\alpha}-1}}{\log\gamma_{n}}\Big(\int_{\frac{e}{\gamma_{n}}}^{1}|\log(v)|\dif u+\int_{1}^{\infty}
\frac{1+|v|^{\beta}}{|v|^{\alpha}}\dif u\Big)=O\big(n^{\frac{1}{\alpha}-1}(\log\gamma_{n})^{-1}\big).
\end{align*}
When $\alpha=1,$ we have by the same argument as the case $\alpha\in(0,1)$ with (\ref{e:HolF'}) and (\ref{e:FraFBou}) replaced by (\ref{e:holdf'c}) and (\ref{e:FraFBouc}), respectively,
\begin{align*}
\mathbf{I}=O\big(\gamma_{n}^{-1}\log \gamma_{n}\big),\quad \mathbf{II}=O\big((\log\gamma_{n})^{-1}\big).
\end{align*}
On the other hand, when $\alpha\in(0,1),$ we have by (\ref{e:FraFBou}) and (\ref{regularity1}),
\begin{eqnarray*}
\left|
\E\big[
(\mathcal{L}^{\alpha,0}f)(\widetilde{S}_{n,i})\big]
-\E\big[
(\mathcal{L}^{\alpha,0}f)(\widetilde{S}_n)\big]
\right|=O(n^{-1}).
\end{eqnarray*}
When $\alpha=1,$ we have by (\ref{e:FraFBouc}) and (\ref{e:FraFHolc}),
\begin{eqnarray*}
\left|
\E\big[
(\mathcal{L}^{1,0}f)(\widetilde{S}_{n,i})\big]
-\E\big[
(\mathcal{L}^{1,0}f)(\widetilde{S}_n)\big]
\right|=O\big(n^{-1}(\log n)^{2}\big).
\end{eqnarray*}
Putting everything together, we get that
\[
d_W(\widetilde{S}_n,Z)=O\big((\log\gamma_{n})^{-1})=O((\log n)^{-1}\big),
\]
which is the desired conclusion.\qed

\section{Proof of Corollary \ref{c:SCLT}}
Assume $h_{0}\in\mathcal H_\be$ and $h_{0}:\mathbb{R}\rightarrow[0,1]$ such that $h_{0}(s)=1$ for $s\leq0,$ $h_{0}(s)=0$ for $s\geq1.$ Fix any $x\in\mathbb{R},$ and define $h(s)=\frac{1}{\psi}h_{0}\big(\psi(s-x)\big)$ for some $\psi\geq1.$ Then, for any $y,z\in\mathbb{R},$ we have
\begin{align*}
|h(y)-h(z)|=\frac{1}{\psi}|h_{0}\big(\psi(y-x)\big)-h_{0}\big(\psi(z-x)\big)|&\leq\frac{1}{\psi}\big(\psi^{\beta}|y-z|^{\beta}\wedge\psi|y-z|\big)\\
&\leq|y-z|^{\beta}\wedge|y-z|,
\end{align*}
which implies $h\in\mathcal H_\be.$ Then,
\begin{align*}
\mathbb{P}(S_{n}\leq x)&\leq\mathbb{E}\big[h_{0}\big(\psi(S_{n}-x)\big)\big]=\psi\mathbb{E}\big[h(S_{n})\big]\\
&\leq\psi\mathbb{E}\big[h(Z)\big]+\psi \dwb\big(S_n,Z\big)\\
&\leq\mathbb{P}(Z\leq x+\frac{1}{\psi})+\psi \dwb\big(S_n,Z\big).
\end{align*}
Furthermore, by boundedness of the stable densities,
$
\mathbb{P}(Z\leq x+\frac{1}{\psi}) - \PP(Z\le x)\leq C_{\alpha}\frac{1}{\psi},
$
which implies
\begin{align*}
\mathbb{P}(S_{n}\leq x)-\mathbb{P}(Z\leq x)\leq\psi \dwb\big(S_n,Z\big)+C_{\alpha}\frac{1}{\psi}.
\end{align*}
Hence, we take $\psi=\big[\dwb\big(S_n,Z\big)\big]^{-\frac{1}{2}},$
\begin{align*}
\mathbb{P}(S_{n}\leq x)-\mathbb{P}(Z\leq x)\leq C_{\alpha}\big[\dwb\big(S_n,Z\big)\big]^{\frac{1}{2}}.
\end{align*}
This gives one half of the claim. The other half follows similarly.

\bibliographystyle{amsplain}

\end{document}